\def\SIAM{0}

\def\LUA{0}

\def\TRUE{1}
\def\FALSE{0}

\ifx\SIAM\TRUE
	\documentclass[final]{siamart220329}
\else
	\documentclass[11pt,reqno]{amsart}
    \ifx\LUA\FALSE
		\usepackage[utf8]{inputenc}
    \fi
\fi

\usepackage{amsmath,amsfonts,amssymb,mathrsfs}
\ifx\SIAM\FALSE
\usepackage{amsthm}
\fi

\ifx\LUA\TRUE
\usepackage{systeme}

\usepackage{mathtools}
\usepackage[
math-style=ISO,
warnings-off={mathtools-colon, mathtools-overbracket},
]{unicode-math}
\else
\usepackage{mathabx}
\fi

\usepackage{microtype}
\DisableLigatures[f]{encoding = *, family = * }
\usepackage{subcaption}
\usepackage{marginnote} 

\ifx\SIAM\FALSE
\numberwithin{equation}{section}
\fi

\usepackage{tabularx}
\usepackage{cite}
\usepackage{enumitem}
\usepackage{graphicx}
\usepackage{tikz}

\usepackage{pgfplots}
\pgfplotsset{compat=1.18}
\usepackage{pgfplotstable}
\usetikzlibrary{through,calc,arrows,intersections,pgfplots.groupplots}
\usepackage{tikz-cd} 

\ifx\LUA\FALSE
\usepackage{bm}
\fi

\usepackage{xcolor}
\newcommand{\rd}{\textcolor{red}}
\newcommand{\bl}{\textcolor{blue}}
\newcommand{\pu}{\textcolor{purple}}

\usepackage{verbatim,booktabs}

\ifx\SIAM\FALSE
\usepackage{lineno}
\usepackage[top=1in, bottom=1in, left=1.0in, right=1.0in]{geometry}
\fi
\allowdisplaybreaks

\ifx\SIAM\FALSE
\usepackage[
	    pdftitle={New Asymptotic Preserving Hybrid Discontinuous {G}alerkin Methods for Radiative Transport Equations with isotropic scattering},
            unicode=true,
            colorlinks,
            linkcolor=red,
            anchorcolor=blue,
            citecolor=green
            ]{hyperref}
\usepackage{cleveref}
\else
\ifpdf
\hypersetup{pdftitle={New Asymptotic Preserving Hybrid Discontinuous {G}alerkin Methods for Radiative Transport Equations with isotropic scattering}}
\fi
\fi

\ifx\SIAM\FALSE
\newtheorem{theorem}{Theorem}[section]
\newtheorem{lemma}[theorem]{Lemma}
\newtheorem{corollary}[theorem]{Corollary}

\theoremstyle{definition}
\theoremstyle{remark}\newtheorem{remark}[theorem]{Remark}
\theoremstyle{definition}\newtheorem{example}{Example}[section]
\theoremstyle{definition}\newtheorem{prob}{Problem}[section]
\newtheorem{assumption}{Assumption}
\else
\newsiamthm{assumption}{Assumption}
\newsiamremark{remark}{Remark}
\newsiamthm{prob}{Problem}
\crefname{prob}{Problem}{Problem}
\newsiamremark{example}{Example}
\crefname{example}{Example}{Example}

\fi

\newcommand{\R}{\mathbb{R}}

\newcommand{\ud}{\,\mathrm{d}}

\ifx\LUA\TRUE
\newcommand{\bx}{\symbfup{x}}
\newcommand{\bn}{\symbfup{n}}
\newcommand{\bomega}{\symbf{\omega}}

\newcommand{\bfm}{\symbf{m}}
\newcommand{\bg}{\symbf{g}}
\newcommand{\bu}{\symbf{u}}
\newcommand{\bv}{\symbf{v}}

\newcommand{\bz}{\symbf{z}}
\newcommand{\bV}{\symbf{V}}
\newcommand{\bbf}{\symbf{f}}
\newcommand{\bA}{\symbf{A}}
\newcommand{\bB}{\symbf{B}}
\newcommand{\bD}{\symbf{D}}
\newcommand{\bR}{\symbf{R}}
\newcommand{\bI}{\symbf{I}}
\newcommand{\bJ}{\symbf{J}}
\newcommand{\bP}{\symbf{P}}
\newcommand{\bQ}{\symbf{Q}}
\newcommand{\bU}{\symbf{U}}
\newcommand{\bW}{\symbf{W}}

\newcommand{\bZ}{\symbf{Z}}
\newcommand{\bLam}{\symbf{\Lambda}}

\NewDocumentCommand{\cpi}{}{\symrm{\pi}}
\else
\newcommand{\bx}{\mathbf{x}}
\newcommand{\bn}{\mathbf{n}}
\newcommand{\bomega}{\bm{\omega}}

\newcommand{\bfm}{\bm{m}}
\newcommand{\bg}{\bm{g}}
\newcommand{\bu}{\bm{u}}
\newcommand{\bv}{\bm{v}}

\newcommand{\bz}{\bm{z}}
\newcommand{\bV}{\bm{V}}
\newcommand{\bbf}{\bm{f}}
\newcommand{\bA}{\bm{A}}
\newcommand{\bB}{\bm{B}}
\newcommand{\bD}{\bm{D}}
\newcommand{\bR}{\bm{R}}
\newcommand{\bI}{\bm{I}}
\newcommand{\bJ}{\bm{J}}
\newcommand{\bP}{\bm{P}}
\newcommand{\bQ}{\bm{Q}}
\newcommand{\bU}{\bm{U}}
\newcommand{\bW}{\bm{W}}

\newcommand{\bZ}{\bm{Z}}
\newcommand{\bLam}{\bm{\Lambda}}

\NewDocumentCommand{\cpi}{}{\mathrm{\pi}}
\fi

\newcommand{\veps}{\varepsilon}

\newcommand*{\tran}{\mathsf{T}}

\newcommand{\vint}[1]{\langle #1 \rangle}
\newcommand{\Vint}[1]{\left\langle #1 \right\rangle}
\newcommand{\avg}[1]{\{\!\!\{#1\}\!\!\}}
\newcommand{\jmp}[1]{[\![#1]\!]}

\newcommand{\qquand}{\qquad \text{and} \qquad}

\DeclareMathOperator{\diag}{diag}

\ifx\SIAM\FALSE
\title[New AP Hybrid DG Methods]{New Asymptotic Preserving, Hybrid Discontinuous {G}alerkin Methods the Radiation Transport Equation with Isotropic Scattering and Diffusive Scaling}
\author[C.~Hauck \and Q.~Sheng]{Cory Hauck$^\S$ \and Qiwei Sheng$^\dag$ \and Yulong Xing$^\sharp$ }

\thanks{$^\S$ Computational and Applied Mathematics Group, Oak Ridge National Laboratory, Oak Ridge, TN 37831 (hauckc@ornl.gov)}
\address[1]{Computational and Applied Mathematics Group, Oak Ridge National Laboratory, Oak Ridge, TN 37831}
\email{hauckc@ornl.gov}
\thanks{$^\dag$ Department of Mathematics, California State University, Bakersfield, CA 93311 (qsheng@csub.edu)}
\address[2]{Department of Mathematics, California State University, Bakersfield, CA 93311}
\email{qsheng@csub.edu}
\thanks{$^\sharp$ Department of Mathematics, Ohio State University, Columbus, OH 43210 (xing.205@osu.edu)}
\address[4]{Department of Mathematics, Ohio State University, Columbus, OH 43210}
\email{xing.205@osu.edu}

\else
\title{New Asymptotic Preserving, Hybrid Discontinuous {G}alerkin Methods the Radiation Transport Equation with Isotropic Scattering and Diffusive Scaling
}
\author{
	Cory D.~Hauck\thanks{Computational and Applied Mathematics Group, Oak Ridge National Laboratory, Oak Ridge, TN 37831 (\email{hauckc@ornl.gov})}
    \and
    Qiwei Sheng\thanks{Department of Mathematics, California State University, Bakersfield, CA 93311 (\email{qsheng@csub.edu})}
    \and
    Yulong Xing\thanks{Department of Mathematics, The Ohio State University, Columbus, OH 43210 (\email{xing.205@osu.edu})}
}
\headers{New AP Heterogeneous and Hybrid DG Methods}{Cory D. Hauck, Qiwei Sheng, and Yulong Xing}
\fi

\begin{document}
\ifx\SIAM\TRUE
\maketitle
\fi

\begin{abstract}
Discontinuous Galerkin (DG) methods are widely adopted to discretize the radiation transport equation (RTE) with diffusive scalings. One of the most important advantages of the DG methods for RTE is their asymptotic preserving (AP) property, in the sense that they preserve the diffusive limits of the equation in the discrete setting, without requiring excessive refinement of the discretization. However, compared to finite element methods or finite volume methods, the employment of DG methods also increases the number of unknowns, which requires more memory and computational time to solve the problems. In this paper, when the spherical harmonic method is applied for the angular discretization, we perform an asymptotic analysis which shows that to retain the uniform convergence, it is only necessary to employ non-constant elements for the degree zero moment only in the DG spatial discretization. Based on this observation, we propose a heterogeneous DG method that employs polynomial spaces of different degrees for the degree zero and remaining moments respectively. To improve the convergence order, we further develop a spherical harmonics hybrid DG finite volume method, which preserves the AP property and convergence rate while tremendously reducing the number of unknowns. Numerical examples are provided to illustrate the effectiveness and accuracy of the proposed scheme. 
\end{abstract}

\ifx\SIAM\FALSE
\maketitle
\else
\fi

\section{Introduction}
The radiation transport equation (RTE) is a kinetic model that describes the flow of radiation through a material medium and plays an important role in a wide range of applications involving neutron \cite{davison1957neutron}, photon \cite{pomraning2005equations}, and neutrino \cite{mezzacappa2020physical} transport processes.
The RTE is an integro-differential equation that is numerically challenging to solve due to its high-dimensional space and the multiscale nature of the solution induced by variations in the material medium.  In particular, the nature of the RTE is hyperbolic when scattering is minimal and diffusive when scattering is very prominent.  In cases where the scattering is prominent throughout the spatial domain of a particular problem, a diffusive approximation \cite{LarsenKeller,bensoussan1979boundary} can be used to approximate the particle density at a significantly reduced computational cost.  
However, in true multiscale problems, the scattering may vary by orders of magnitude, requiring the use of the kinetic description provided by the RTE.  Unfortunately, standard finite difference and finite volume (FV) methods designed for hyperbolic problems can become computationally expensive in scattering-dominated regimes.  For stability and accuracy reasons, they must resolve the mean free path between particle collisions with the material.  Thus care must be taken to construct numerical methods which accurately capture the asymptotic diffusion equation that arises in the limit of infinite scattering \cite{LMM1987,LM1989,lowrie2002methods,jin1991discrete,hauck2009temporal}.  Such methods are often referred to as asymptotic preserving (AP) \cite{jin2010asymptotic,Jin1999}.

The discretization of the RTE is usually carried out separately for the angular, spatial, and temporal variables.  In this paper, we focus on the spherical harmonics ($P_N$) method for angular discretization. The $P_N$ method \cite{CZ1967} is one of the most widely employed angular discretizations due to its robustness, accuracy, and computational efficiency. It is a spectral method in the angle that seeks an approximate solution as a sum of spherical harmonics basis functions with coefficients that depend on space (and time in the time-dependent setting).  These coefficients correspond to moments of the RTE; thus the $P_N$ equations can be viewed as a type of moment model \cite{garrett2013comparison}. A major advantage of the $P_N$ method is that it diagonalizes any scattering operator with a rotationally invariant kernel \cite{LM1984}, thereby reducing the cost of evaluating the operator.  In addition,  the method exhibits spectral convergence for smooth problems \cite{frank2016convergence}, and preserves rotational invariance properties of the RTE \cite{garrett2016eigenstructure}.  Moreover, for steady-state problems, the $P_1$ ($N=1$) method is equivalent to the first-order form of the diffusion limit, which makes the asymptotic analysis relatively simple.

For spatial discretization, we focus on discontinuous Galerkin (DG) finite element methods, which become a popular choice in solving the RTE due to their robustness, flexibility, and high accuracy.  Importantly, with sufficiently rich spaces, these methods recover the interior diffusion limit%
\footnote{Here the \textit{interior diffusion limit} refers to the correct interior discretization.  To fully recover the limit, correct boundary conditions are also required.  Boundary conditions are not addressed here.}
of the RTE \cite{LM1989, Adams2001, GK2010, McClarren20102290}.  Roughly speaking, the necessary requirement for DG methods with upwind fluxes to achieve such property is that the finite element spaces support global linear polynomials.  This requirement translates to local $\mathbb{P}_1$ and $\mathbb{Q}_1$ approximations for triangular and rectangular elements, respectively.  The $\mathbb{Q}_1$ requirement can be prohibitive, particularly in three dimensions where the added cost over an FV method increases by a factor of eight. As a consequence, while the AP property of the DG method is desirable, it comes at the expense of higher computational costs, including increased memory demands and longer computing times.

Inspired by the analysis in \cite{mcclarren2008effects, GK2010}, we propose a $P_N$-DG method with heterogeneous polynomial spaces for different moments that maintains the AP property of the DG method while significantly reducing the number of unknowns. Specifically, only the degree zero moment requires a $\mathbb{Q}_1$ approximation in space, while the remaining moments can be approximated in space using piece-wise constants. As a result, the computational cost becomes comparable to that of an FV method. However, even though this method captures the interior diffusion limit, it is generally first-order in space. To alleviate this limitation, we propose a hybrid approach that combines the DG method with $\mathbb{Q}_1$ elements for the degree zero moment and a second-order FV method for the remaining moments. Numerical experiments confirm that this hybrid method successfully captures the interior diffusion limit and achieves uniform second-order accuracy.

The concept of employing different spatial discretizations for different components of the RTE has been explored in several other contexts.  In \cite{buet2015asymptotic}, the $P_N$ equations were split at the continuum level into a diffusive ($P_1$) and non-diffusive part; each part was then discretized with a different FV method.  In \cite{sun2020low}, a discrete ordinate DG ($S_N$-DG) method was proposed in which each ordinate is assigned a unique average but shares a common slope.  In \cite{heningburg2021hybrid,hauck2013collision}, the RTE is separated into collided and uncollided components, and different discretizations are employed for the angular and spatial variables of each component.

The rest of this article is organized as follows. In \Cref{sec:settings}, we introduce the scaled RTE with isotropic scattering and its spherical harmonic discretization along with the notations used throughout the paper. A priori estimates regarding the solutions of the spherical harmonic equation are also presented. In the last part of this section, we define the DG scheme for the spherical harmonic equation and prove its stability and well-posedness. In \Cref{sec:asym_ana}, the asymptotic analysis of the approximation is performed and a DG method with heterogeneous polynomial spaces for different moments is proposed. 
A hybrid DG/FV method is proposed in \Cref{sec:err_ana_hybrid}. Numerical experiments are presented in \Cref{sec:num_results}, and concluding remarks are given in \Cref{sec:conclusion}. 

\section{Preliminaries and problem setting}\label{sec:settings}
Throughout the paper, 
we adopt the conventional notation $H^r(D)$ to indicate Sobolev spaces on (possibly lower-dimensional) subdomain $D\subset X$ with the norm $\|\cdot\|_{r,D}$. 
Clearly, $H^0(D)=L^2(D)$, and this norm is denoted by $\|\cdot\|_D$. 
\subsection{Scaled radiative transfer/transport equations} 
Let $X =[0,1)^3$ be the unit torus in three dimensions, and let $\mathbb{S}$ be the unit sphere in $\mathbb{R}^3$.
%
We consider the following scaled versions of the steady-state and time-dependent RTEs with periodic boundary conditions for the unknown angular flux $u = u(\bx,\bomega)$ and $\psi = \psi(t,\bx,\bomega)$, respectively:
	\begin{alignat}{2}
		\bomega\cdot\nabla u(\bx,\bomega)+\frac{\sigma_{\mathrm{t}}}{\veps}u(\bx,\bomega)
		&=\left(\frac{\sigma_{\mathrm{t}}}{\veps}-\veps\sigma_{\mathrm{a}}\right)\Vint{u}(\bx)+\veps f, 
        &&\quad (\bx,\bomega) \in X \times \mathbb{S},\label{eq:rte_scale_all} 
	\end{alignat}
and
\begin{subequations}\label{eq:rte_scale_time_all}
\begin{alignat}{3}
	\veps\partial_t \psi(t,\bx,\bomega) + \bomega\cdot\nabla \psi(t,\bx,\bomega) 
         &+\frac{\sigma_{\mathrm{t}}}{\veps}\psi(t,\bx,\bomega) &&& \nonumber \\ 
	&=\left(\frac{\sigma_{\mathrm{t}}}{\veps}-\veps\sigma_{\mathrm{a}}\right)\Vint{\psi}(t,\bx)+\veps f,
    &&\quad (\bx,\bomega) \in X \times \mathbb{S}, t>0 \label{eq:rte_scale_time}\\
	\psi(0,\bx,{\bomega})
     &=\psi_0(\bx,\bomega),
     &&\quad (\bx,{\bomega})\in X\times\mathbb{S},\label{eq:rte_scale_time_init}
\end{alignat}
\end{subequations}
where $\Vint{v}=\int_{\mathbb{S}} v \ud{\bomega} \,/ \int_{\mathbb{S}}\ud\bomega$ is the average of $v$ over $\mathbb{S}$
and \(\int_{\mathbb{S}}\ud\bomega=4\pi\).
The functions $\Vint{u}$ and $\Vint{\psi}$ are called the scalar flux.

The parameter $\veps$ is a scaling constant that characterizes the amount of scattering in the problem. For simplicity, we assume $\veps \in (0,1)$. The solutions $u$ and $\psi$ of \eqref{eq:rte_scale_all} and \eqref{eq:rte_scale_time_all}, respectively, depend on $\veps$. However, to reduce the complexity of the notations, $u$ (or $\psi$) instead of $u^{\veps}$ (or $\psi^{\veps}$) will be employed except in \Cref{sec:asym_ana}. 

The functions $\sigma_{\mathrm{a}} = \sigma_{\mathrm{a}}(\bx)$ and $\sigma_{\mathrm{t}}=\sigma_{\mathrm{t}}(\bx)$ are (known) macroscopic absorption and total cross sections, respectively, and $f = f(\bx,\bomega)$ is a (known) source. 
The scaled scattering cross-section $\sigma_{\mathrm{s},\veps}$ takes the form
\begin{equation}
	\sigma_{\mathrm{s},\veps} =  \sigma_{\mathrm{t}} - \veps^2\sigma_{\mathrm{a}}.
\end{equation}
Here and below, a function is said to be isotropic if it is constant with respect to the angular variable $\bomega$.

We make the following assumptions:  
\begin{subequations}\label{as:rte_scale}
	\begin{align}
		&\sigma_{\mathrm{t}},\sigma_{\mathrm{a}}\in L^\infty(X), \\
  &\sigma_{\mathrm{s},\veps} = \sigma_{\mathrm{t}}-\veps^2\sigma_{\mathrm{a}}\geq\sigma_{\mathrm{s}}^{\mathrm{min}} \text{ in } \bx \in X \text{   for a constant } \sigma_{\mathrm{s}}^{\mathrm{min}} >0, \label{as:rte_scale_1}\\
        & \sigma_{\mathrm{a}} \ge \sigma_{\mathrm{a}}^{\mathrm{min}} \text{ in } X \text{   for a constant } \sigma_{\mathrm{a}}^{\mathrm{min}} >0,\label{as:rte_scale_2}\\
		&f(\bx,\bomega)\in L^2(X\times\mathbb{S}).\label{as:rte_scale_3}
	\end{align}
\end{subequations}
The condition \eqref{as:rte_scale_2}, while not strictly necessary, is often used in a priori estimates.

The limit $\veps\rightarrow 0$ corresponds to the ratio of the mean free path of the particles to the diameter of $X$ going to zero. In this regime, the media becomes optically thick and is dominated by scattering.
It can be shown that with isotropic initial conditions, the solutions $u$ and $\psi$ in the diffusion limit (i.e., when $\veps\rightarrow 0$) satisfy the following diffusion equations, respectively \cite{LarsenKeller,habetler1975uniform,bensoussan1979boundary}:
	\begin{align}\label{eq:diffusion}
		\nabla\cdot \left(\frac{1}{3\sigma_{\mathrm{t}}(\bx)}\nabla u(\bx)\right) + \sigma_{\mathrm{a}}(\bx)u(\bx) &= \vint{f}(\bx),\quad \hspace{0.3cm}\bx\in X,
	\end{align}
and 
\begin{subequations}
	\begin{align}
		\partial_t \psi(t,\bx) + \nabla\cdot \left(\frac{1}{3\sigma_{\mathrm{t}}(\bx)}\nabla \psi(t,\bx)\right) + \sigma_{\mathrm{a}}(\bx)\psi(t,\bx) &= \vint{f}(\bx), \; (t,\bx)\in [0,T]\times X,\\
		\psi(0,\bx)&=\psi_0(\bx),\quad \bx\in X.
	\end{align}
\end{subequations}

In the current work, we focus on periodic domains which can be easily treated in a variational framework \cite{ES2012,MRS1999}.  More general boundary conditions are also possible but may be computationally inefficient due to the evaluation of half-space integrals that lead to dense matrices during assembly.  An alternative approach \cite{egger2019perfectly,SW2021} based on perfectly matched layers introduces a small modification that results in a negligible error in consistency but improves the sparsity of the matrix representation.  With this fact in mind, we focus only on the periodic boundary conditions, although the challenge of recovering at the numerical level the correct boundary condition in the diffusion limit remains open.

\subsection{Equations in reduced geometry}
\label{subsec:reduced-geometry}
While the analysis here is often applied to the full 3-D setting, the numerical results in \Cref{sec:num_results} are applied to reduced models with certain symmetries:  a $1$-D slab geometry model and a $2$-D plane-parallel model. 
These (steady-state) reduced models \cite{agoshkov1998boundary,Modest} are briefly introduced here. The corresponding time-dependent equations are obvious and therefore omitted.  

In a spherical coordinate system, $\bomega$ is determined by the polar angle $\theta\in[0,\pi]$ and azimuthal angle $\varphi\in[0,2\pi)$.  Thus 
\begin{equation}
\label{eq:omega_components}
    \bomega=
    \begin{bmatrix} \omega_x & \omega_y & \omega_z \end{bmatrix}^{\tran}
    =\begin{bmatrix} \sin\theta\cos\varphi & \sin\theta\sin\varphi & \cos\theta \end{bmatrix}^{\tran}.
\end{equation}
Reduced geometry problems will frequency be expressed in terms of these angles, as well as the variable $\mu = \cos \theta \in [-1,1]$.

\paragraph*{\textbf{One-dimensional slab geometry problem}} 
Let \(I=[0,1)\) be the unit \(1\)-torus in \(\R^2\).
In the slab geometry, the RTE \eqref{eq:rte_scale_all} reduces to the following form 
	\begin{equation}
		\mu\frac{\partial u}{\partial z}+\frac{\sigma_{\mathrm{t}}}{\veps} u = \frac{1}{2}\left(\frac{\sigma_{\mathrm{t}}}{\veps}-\veps\sigma_{\mathrm{a}}\right) \int^{1}_{-1} u(z,\mu')\ud \mu' +  f,\quad \forall z\in I,\label{eq:1d}
	\end{equation}
where $z \in I$, $u=u(z,\mu)$, $f=f(z,\mu)$, and $\ud \mu'$ is the Lebesgue measure on $(-1,1)$. 

\paragraph*{\textbf{Two-dimensional plane-parallel model}}
Let \(J=[0,1)^2\) be the unit \(2\)-torus in \(\R^3\). In the parallel-plane model, the RTE \eqref{eq:rte_scale_all} takes the form
	\begin{equation}
		\sqrt{1-\mu^2}\cos\varphi\frac{\partial u}{\partial x} + \mu\frac{\partial u}{\partial z} + \frac{\sigma_{\mathrm{t}}}{\veps}u
		=\left(\frac{\sigma_{\mathrm{t}}}{\veps}-\veps\sigma_{\mathrm{a}}\right)\Vint{u}+\veps f,\label{eq:2d}
	\end{equation}
where \((x,z)\in J\), $u=u(x,z,\mu,\varphi)$, and $f=f(x,z)$.

\subsection{Time discretion and translation to steady-state form}\label{sec:Time_disc}
There are a variety of schemes to discretize the time variable in the RTE \eqref{eq:rte_scale_time}. Implicit methods are common, and here we employ the backward differentiation formula (BDF). In detail, let  $t_n=n\varDelta t$ with $\Delta t$ being the time step size, and $\psi^{(n)} (\bx,\bomega)=\psi(t_n,\bx,\bomega)$. Then the $s$-step BDF method applied to the time-dependent RTE \eqref{eq:rte_scale_time} can be formed as 
\begin{multline}\label{eq:rte_scale_time_disc}
	\veps \psi^{(n)} = \veps\sum_{p=1}^{s}\beta_{s,p} \psi^{(n-p)} + \gamma_s\varDelta t  \bigg(-\bomega\cdot\nabla \psi^{(n)} - \frac{\sigma_{\mathrm{t}}(\bx)}{\veps}\psi^{(n)} \\
	+ \left(\frac{\sigma_{\mathrm{t}}(\bx)}{\veps} - \veps\sigma_{\mathrm{a}}(\bx)\right)\Vint{\psi^{(n)}} + \veps f(\bx,\bomega)\bigg),
\end{multline}
where the spatial and angular variables remain continuous. 
The coefficients of $\beta_{s,p}$ and $\gamma_s$ for $s$-step BDFs with $p\le s$ can be found, e.g., in \cite{SM2003} for $s \le 5$. 

After some simple manipulations, the steps in \eqref{eq:rte_scale_time_disc} can be rewritten as the following steady-state RTE
\begin{equation}
\label{eq:steady-state-equiv}
	\bomega\cdot\nabla \psi^{(n)}+\frac{\sigma^{(n)}_{\mathrm{t}}}{\veps}\psi^{(n)}
	=\left(\frac{\sigma_{\mathrm{t}}}{\veps}-\veps\sigma_{\mathrm{a}}\right)\Vint{\psi^{(n)}} +\veps f^{(n)},
\end{equation}
where 
\begin{equation}
	f^{(n)} = f + \frac{1}{\gamma \varDelta t} \sum_{p=1}^{s}\beta_{s,p} \psi^{(n-p)}
\end{equation}
is a source depending on the values at previous steps and $\sigma^{(n)}_{\mathrm{t}}=\sigma_{\mathrm{t}}+\frac{\veps^2}{\gamma_s\varDelta t}$ is the effective total cross-section.
After dropping the temporal index $n$, \eqref{eq:steady-state-equiv} takes the form of the scaled steady-state RTE \eqref{eq:rte_scale_all}, which will be the focus of the analysis in the rest of this paper.

 
\subsection{Spherical harmonic method for angular discretization}\label{sec:SH_PN}
In this section, we briefly review the spherical harmonic functions and the $P_N$ method. The interested readers are referred to \cite{AH2012,Claus1966} for more discussions.  Formulas for the spherical harmonics are given in \Cref{sec:appendix_PN}. 


Given integers $\ell\ge 0$ and $\kappa\in [-\ell,\ell]$, let $m^\kappa_\ell$ be the real-valued spherical harmonic of degree $\ell$ and order $\kappa$ on $\mathbb{S}$, normalized such that $\int_{\mathbb{S}} (m^\kappa_\ell(\bomega))^2 d \bomega = 1$.
We collect the $n_\ell := 2\ell+1$ real-valued normalized harmonics of degree $\ell$ together into a vector-valued function $\bfm_\ell=\begin{bmatrix} m^{-\ell}_\ell & m^{-\ell+1}_\ell & \cdots & m^0_\ell & \cdots & m^{\ell-1}_\ell & m^\ell_\ell \end{bmatrix}^{\tran}$ and then for any given $N$, we set $\bfm = \begin{bmatrix}\bfm_0^{\tran} & \bfm_1^{\tran} & \cdots & \bfm_N^{\tran}\end{bmatrix}^{\tran}$. 
In all, $\bfm$ has $L := \sum^N_{\ell=0} n_\ell= (N+1)^2$ components which form an orthogonal basis for the space
$\mathbb{P}_N=\left\{\sum^N_{\ell=0} \sum^\ell_{\kappa=-\ell} c^\kappa_\ell m^\kappa_\ell\colon c^\kappa_\ell\in \mathbb{R}, \text{ and } 0\le\ell\le N, |\kappa|\le\ell\right\}$.
Furthermore, the spherical harmonics fulfill a recursion relation of the form \cite{AH2012}
\begin{equation}\label{eq:SH_rec}
\omega_i \bfm_\ell = \bA^{(i)}_{\ell,\ell+1}\bfm_{\ell+1} + \bA^{(i)}_{\ell,\ell-1} \bfm_{\ell-1},
\end{equation}
where $\bA^{(i)}_{\ell,\ell'} = \int_{\mathbb{S}} \omega_i\bfm_\ell\bfm^{\tran}_{\ell'}\ud\bomega$ and $\big(\bA^{(i)}_{\ell,\ell'}\big)^{\tran}=\bA^{(i)}_{\ell',\ell}$.

The $P_N$ equations for the radiative transfer equation \eqref{eq:rte_scale_all} are derived by approximating $u$ by a function $u_{\mathrm{PN}}$ of the form of
\begin{equation}\label{eq:u_appx}
u_{\mathrm{PN}}(\bx,\bomega) = \bu^{\tran}\bfm = \sum_{\ell=0}^N\bu^{\tran}_\ell\bfm_\ell = \sum_{\ell=0}^N\sum_{|k|\leq \ell} {u}^k_\ell{m}_\ell^k
\end{equation}
such that $\forall\, v\in \mathbb{P}_N$,
\begin{equation}\label{eq:rte_scale_pre_Pn}
\Big\langle v\big(\bomega\cdot\nabla u_{\mathrm{PN}}(\bx,\bomega)+\frac{\sigma_{\mathrm{t}}(\bx)}{\veps} u_{\mathrm{PN}}(\bx,\bomega) -\left(\frac{\sigma_{\mathrm{t}}(\bx)}{\veps}-\veps\sigma_{\mathrm{a}}(\bx)\right)\bar{u}_{\mathrm{PN}}\big) \Big\rangle = \veps\langle v f \rangle.
\end{equation}
Because the components $m_l^\kappa$ are normalized, it follows that
$u_\ell^k := \int_{\mathbb{S}}m_\ell^k u_{\mathrm{PN}} \ud\bomega,  \quad \ell=0,\cdots,N, \; |k|\le\ell$.
The double index notation for $\bu$ is inherited from $\bfm$; that is 
$\bu_\ell = \begin{bmatrix}u_\ell^{-\ell}&\cdots&u_\ell^\ell\end{bmatrix}^{\tran}$ and $\bu= \begin{bmatrix}\bu_0^{\tran} & \bu_1^{\tran} & \cdots & \bu_N^{\tran}\end{bmatrix}^{\tran}$. However, for convenience, we often use a single index to denote the components in $\bfm$ and $\bu$, e.g., $\bu=\begin{bmatrix}u_1 & u_2 & \cdots & u_L\end{bmatrix}^{\tran}$, in which case the function $u_{\mathrm{PN}}$ can be expressed as
$
u_{\mathrm{PN}} =\sum_{i=1}^{L}m_iu_i.
$

Setting $v=\bfm$ in \eqref{eq:rte_scale_pre_Pn} and using \eqref{eq:u_appx} and the fact that $\int_{\mathbb{S}} \bfm\bfm^{\tran}\ud\bomega=\bI$, we can reformulate \eqref{eq:rte_scale_pre_Pn} into a system of hyperbolic equations ($P_N$ equations): 
	\begin{align}
		\bA\cdot \nabla \bu(\bx) + \veps\sigma_{\mathrm{a}}\bu(\bx) + \left(\frac{\sigma_{\mathrm{t}}}{\veps}-\veps\sigma_{\mathrm{a}}\right) \bR\bu(\bx) &= \veps\bbf(\bx) \qquad \text{ in } X, \label{eq:rte_scale_Pn}
	\end{align}
Here $\bu(\bx)\in \mathbb{R}^L$;
$\bbf=\int_{\mathbb{S}} \bfm f \ud \bomega$; $\bR=\bI-\diag(1,0,\cdots,0)=\diag(0,1,\cdots,1)$  is diagonal and positive semi-definite;
and the dot product between 
\begin{equation}
    \bA:=\begin{bmatrix} \left(\bA^{(1)}\right)^\tran & \left(\bA^{(2)}\right)^\tran & \left(\bA^{(3)}\right)^\tran \end{bmatrix}^\tran
\end{equation} and the gradient is understood as
$\bA\cdot \nabla := \sum^3_{i=1}\bA^{(i)}\partial_i$ with 
$\bA^{(i)}=\int_{\mathbb{S}}\omega_i\bfm\bfm^{\tran}\ud\bomega$,  $i=1,2,3$.
Note that $\bA^{(i)},i=1,2,3$ is symmetric and sparse. Moreover, the recursion relation \eqref{eq:SH_rec} and orthogonality conditions for spherical harmonics \cite{AH2012, Claus1966} imply that $\bA^{(i)}_{\ell,\ell'}$ is nonzero only if $\ell' = \ell \pm 1$. Therefore, 
\begin{equation}\label{eq:A_struc}
\bA^{(i)} = 
\begin{bmatrix}
0 				& \bA^{(i)}_{0,1} 	& 0             	& 0 				& \dots					& 0  \\
\bA^{(i)}_{1,0} & 0 			  	& \bA^{(i)}_{1,2}	& 0 				& \dots 			& 0\\
0				& \bA^{(i)}_{2,1}  	& 0 			   	& \bA^{(i)}_{2,3}   & \dots				& 0\\
0				& 0					& \ddots			& \ddots			& \ddots			& 0\\
\vdots			& \vdots			& \ddots 			&\bA^{(i)}_{N-1,N-2}& 0					&\bA^{(i)}_{N-1,N} \\
0				& 0					& \dots				&0					&\bA^{(i)}_{N,N-1}& 0 \\
\end{bmatrix}.
\end{equation}
Since $\bA^{(i)}$, $i=1,2,3$ are symmetric, they can be diagonalized as:
\begin{equation}\label{eq:A_sym}
\bA^{(i)}=\bU_i\bLam^{(i)} \bU_{i}^{\tran},
\end{equation}
where $\bU_i$ is a real-valued orthogonal matrix and $\bLam^{(i)}=\diag\left(\lambda^{(i)}_1,\lambda^{(i)}_2,\cdots,\lambda^{(i)}_L\right)$ is a real-valued diagonal matrix.
With this decomposition, let
\begin{equation}
\label{eq:matrix_abs}
|\bA^{(i)}|=\bU_i |\bLam^{(i)}| \bU_i^{\tran}
\quad \text{where} \quad
|\bLam^{(i)}|:=\diag(|\lambda^{(i)}_1|,|\lambda^{(i)}_2|,\cdots,|\lambda^{(i)}_L|)   
 \end{equation}
define the absolution value of such matrices.

Let $D\subseteq X$ be a (possibly lower-dimensional) subdomain of $X$. Given $\bu, \bv\in \left[L^2(D)\right]^L$,  define the inner product
\begin{equation}
(\bu,\bv)_D=\int_D \bu^{\tran} \bv\ud\bx = \sum_{\ell=0}^N\sum_{|k|\leq \ell} \int_D u^k_\ell\,v^k_\ell \ud \bx,
\end{equation}
and the Sobolev norms
\begin{equation}\label{eq:norm_def}
\|\bu\|_{r,D}=\left(\sum_{\ell=0}^N\sum_{|k|\leq \ell}\|u^k_\ell\|_{r,D}^2\right)^{1/2}, \quad\text{ for } \bu\in \left[H^r(D)\right]^L.
\end{equation}
When $r=0$ and $D=X$, we omit the subscripts $0$ and $X$, i.e., $(\bu,\bv):=(\bu,\bv)_X$ and $\|\bu\|:=\|\bu\|_{0,X}$. Since $\int_{\mathbb{S}}\bfm\bfm^{\tran}\ud\bomega=\bI$, it follows that
\begin{equation}
\|u_{\mathrm{PN}}\|_{L^2(X\times\mathbb{S})} 
=\|\bu\|.
\end{equation}
Define the space
$\bV=\left\{\bu\in [L^2(X)]^{L}\colon \bA\cdot\nabla\bu\in [L^2(X)]^{L}\right\}$
with the associated norm 
$\|\bu\|_{\bV} = \|\bu\| + \|\bA\cdot\nabla\bu\|$,
and the space $\bW=\left\{\bu\in\bV:\bu\right.$ is $1$-periodic in each spatial argument $x_i$, $\left.i=1,2,3\right\}$. We have the following existence and uniqueness result \cite[Theorem 6]{SW2021}.
\begin{theorem}\label{th:Pn_wellposed}
	Assume the assumptions \eqref{as:rte_scale} hold. For any fixed $\veps>0$, the system of $P_N$ equations \eqref{eq:rte_scale_Pn} has a unique solution $\bu\in\bW$.
\end{theorem}

\subsection{Discontinuous {G}alerkin method for spatial discretization}\label{sec:DG}
   Let  $\mathcal{T}^h$ be a regular family of partition of $X$ with elements $K$.
The meshes on the boundary are assumed to be periodic to avoid unnecessary technicalities, i.e., the surface meshes on two opposite parallel faces of $\partial X$ are identical.
Define $h_K:= \textrm{diam}(K)$ and $h:=\max_{K\in \mathcal{T}^h}h_K$.
Denote by $\bn^K=\begin{bmatrix}n_1^K & n_2^K & n_3^K\end{bmatrix}^{\tran}$ the unit outward normal to $\partial K$ with respect to $K$. 

Let  $V_{h,k}$ be a DG finite element space whose elements are polynomials of degree no more than $k$ when restricted to any element $K$, i.e.,
\begin{equation}\label{eq:poly_space}
	V^{h,k}=\begin{cases}
		\{v\in L^2(\mathbb{S})\colon v|_K\in \mathbb{P}_k(K)\;\forall K\in \mathcal{T}^h\}, &\text{ if } \mathcal{T}^h \text{ is triangular}, \\
		\{v\in L^2(\mathbb{S})\colon v|_K\in \mathbb{Q}_k(K)\;\forall K\in \mathcal{T}^h\}, &\text{ if } \mathcal{T}^h \text{ is rectangular},
	\end{cases}
\end{equation}
where $k$ is a nonnegative integer, $\mathbb{P}_k(K)$ denotes the set of all polynomials on $K$ with maximum total degree $k$, and
$\mathbb{Q}_k(K)$ denotes the set of all polynomials on $K$ with maximum degree $k$ in each dimension.
Define $\bV^{h,k_\ell}_\ell=[V^{h,k_\ell}]^{2\ell+1}$, $\ell=0,1,\cdots,N$,
and the Cartesian product
$\bV^h=  \bV^{h,k_0}_0\times \bV^{h,k_1}_1\times \cdots \times \bV^{h,k_N}_N$. 

Given an element $K\in\mathcal{T}^h\subset \mathbb{R}^3$, and any function $\bg\in \bV^h$, 
we define the inside ($-$) and outside ($+$) values of a function $\bg$ (with respect to $K$) as, respectively,
\begin{subequations}
	\begin{align}
	\bg^\pm (\bx) &=\lim\limits_{\epsilon \to 0}\bg(\bx \pm \epsilon\,\bn^K(\bx)), \quad \;\;\;\forall \bx\in \partial K.
	\end{align}
\end{subequations}
The average $\avg{\bg}$ and jump $\jmp{\bg}$ in $\bg\in \bV^h$ with respect to $K$ are given by
\begin{equation}
    \label{eq:avg_jmp}
    \avg{\bg}=\frac12(\bg^+ + \bg^-)
    \qquand
   \jmp{\bg}=\bg^+-\bg^-.
\end{equation}
%

For any $\bu^h,\bv^h\in \bV^h$, define
\begin{subequations}\label{eq:a_b}
	\begin{align}\label{eq:a}
	\mathfrak{a}^h(\bu^h,\bv^h) 
    &= \sum_{K\in\mathcal{T}^h}\bigg\{ \Big(\bn^K\cdot\overrightarrow{\bA\bu^h}, \, \bv^{h-}\Big)_{\partial K} - \sum_{i=1}^3(\bA^{(i)}\bu^h, \partial_i \bv^h)_K + (\bQ\bu^h,\bv^h)_K \bigg\}, \\
	\mathfrak{f}(\bv^h) 
 &= \veps\sum_{K\in\mathcal{T}^h}(\bbf,\bv^h)_K
	\end{align}
\end{subequations}
where  
\begin{equation}\label{eq:def_Q}
	\bQ = \veps\sigma_{\mathrm{a}}\bI + \left(\frac{\sigma_{\mathrm{t}}}{\veps}-\veps\sigma_{\mathrm{a}}\right) \bR = \begin{bmatrix}\veps\sigma_{\mathrm{a}} &  \\  & \frac{\sigma_{\mathrm{t}}}{\veps}\bI_{L-1} \end{bmatrix} 
\end{equation}
and $\bI$ is the identity matrix.  Following \cite{HLL1983}, we define the numerical flux 
as
\begin{equation}\label{eq:flux_Pn_comp}
    n_i^K\overrightarrow{\bA^{(i)}\bu^h} = n_i^K\bA^{(i)}\{\!\!\{\bu^h\}\!\!\} -\frac{1}{2}\left|n_i^K\right|\,\bD^{(i)}\,[\![\bu^h]\!], \quad i=1,2,3,
\end{equation}
and therefore
\begin{equation}\label{eq:flux_Pn}
	\bn^K\cdot\overrightarrow{\bA\bu^h} = \bn^K\cdot\bA\{\!\!\{\bu^h\}\!\!\} -\frac{1}{2}|\bn^K|\cdot\bD\,[\![\bu^h]\!],	
\end{equation}
where $|\bn^K|\cdot\bD = \sum_{i=1}^3 |n_i^K|\bD^{(i)}$ and, assuming the upwind flux, $\bD^{(i)}=|\bA^{(i)}|$. 
The fully discrete spherical harmonic DG scheme for \eqref{eq:rte_scale_all}  can be cast as follows.
\begin{prob}[Spherical harmonic DG scheme]\label{prob:SH_DG}
	Find $\bu^h\in \bV^h$ such that
	\begin{equation}\label{eq:bilinear}
	\mathfrak{a}^h(\bu^h,\bv^h) = \mathfrak{f}(\bv^h) \quad \forall \bv^h\in \bV^h, 
	\end{equation}
 where $\mathfrak{a}^h$ and $\mathfrak{f}$ are defined in \eqref{eq:a_b}.
\end{prob}

To study the stability and well-posedness of \Cref{prob:SH_DG},
we define the norm
\begin{equation}\label{def:norm}
\|\bv^h\|_{\mathfrak{a}^h} = \left(\frac{1}{4}\sum_{K\in\mathcal{T}^h} \Big(|\bn^K|\cdot\bD\,[\![\bv^h]\!], [\![\bv^h]\!]\Big)_{\partial K} + (\bQ\bv^h,\bv^h)\right)^\frac{1}{2},
\end{equation}
for any $\bv^h\in \bV_{h}$.

\begin{lemma}[Stability]\label{lem:stab}
	Under the assumptions \eqref{as:rte_scale_1} and \eqref{as:rte_scale_2}, 
	\begin{equation}\label{eq:stab}
	\|\bv^h\|^2_{\mathfrak{a}^h} = \mathfrak{a}^h(\bv^h,\bv^h), \quad \forall \bv^h\in \bV^h.
	\end{equation}
\end{lemma} 
\begin{proof}
	By the definition \eqref{eq:a}, it follows that for any $\bv^h\in \bV^h$, $\mathfrak{a}^h(\bv^h,\bv^h) = \mathrm{I} +\mathrm{II} + \mathrm{III}$, where
	\begin{align*}
	\mathrm{I} &:=  \sum_{K\in\mathcal{T}^h} \sum_{i=1}^3 \Big(n_i^K\bA^{(i)}\{\!\!\{\bv^h\}\!\!\} -\frac{1}{2}|n_i^K|\bD^{(i)}[\![\bv^h]\!], \bv^{h-}\Big)_{\partial K}, \\
	\mathrm{II} &:=  -\sum_{K\in\mathcal{T}^h} \sum_{i=1}^3 (\bA^{(i)}\bv^h,\partial_i \bv^h)_K, \quad
	\mathrm{III} :=  \sum_{K\in\mathcal{T}^h} (\bQ\bv^h,\bv^h)_K.
	\end{align*}
	For the first two terms, we have
	\begin{equation}
	\mathrm{I}
	=  \sum_{K\in\mathcal{T}^h}  \frac{1}{2}\Big(\bn^K\cdot\bA \bv^{h-}, \bv^{h-}\Big)_{\partial K} + \sum_{K\in\mathcal{T}^h} \frac{1}{4} \Big(|\bn^K|\cdot\bD\,[\![\bv^h]\!], [\![\bv^h]\!]\Big)_{\partial K}.\nonumber
	\end{equation}
	\begin{equation}
	\mathrm{II}
	= -\sum_{K\in\mathcal{T}^h} \sum_{i=1}^3 \frac{1}{2} (n_i^K\bLam^{(i)} \bU_i^{\tran}\bv^{h-}, \bU_i^{\tran}\bv^{h-})_{\partial K} 
	= -\sum_{K\in\mathcal{T}^h} \frac{1}{2} (\bn^K\cdot\bA\bv^{h-}, \bv^{h-})_{\partial K}.  \nonumber
	\end{equation}
	Therefore, we have $\mathrm{I} + \mathrm{II} = \sum_{K\in\mathcal{T}^h} \frac{1}{4} \left(|\bn^K|\cdot\bD\,[\![\bv^h]\!], [\![\bv^h]\!]\right)_{\partial K}$ and \eqref{eq:stab} follows.
\end{proof}
Since $\bV^h$ is of finite dimension, all norms in $\bV^h$ are equivalent, which ensures continuity of the bilinear form. Due to the Lax–Milgram lemma, under the assumptions \eqref{as:rte_scale_1} and \eqref{as:rte_scale_2}, the spherical harmonic DG method \eqref{eq:bilinear} has a unique solution.

\section{Asymptotic Analysis and Diffusion Limit}\label{sec:asym_ana}
In this section, we first perform the asymptotic analysis of the spherical harmonic DG scheme \cref{eq:bilinear} under the assumption $\veps\rightarrow 0$, and discuss the approximation to the diffusion limit. Based on these observations, we then design a heterogeneous DG method, which can achieve the AP property while with less degree of freedom.

\subsection{Asymptotic analysis}
The main result of this section is the following theorem on proving the AP property.
The proof follows the basic strategy in \cite{GK2010}, where instead the discrete-ordinate method was employed to discretize the angular variable. We remark that to indicate the implicit dependence of $\veps$, we start to use $\bu^{\veps,h}$ instead of $\bu^h$ to denote the solution for the analysis in this section. Similarly, $\bu^{\veps,h}_\ell$ is used to denote the $\ell$th moment of $\bu^{\veps,h}$. 

Define
\begin{equation}\label{eq:def_phi_J}
		\phi^{\veps,h} =\int_{\mathbb{S}} u_{\mathrm{PN}}\ud\bomega \quad \text{ and } \quad 
    \bJ^{\veps,h} =\frac{1}{\veps}\int_{\mathbb{S}}\bomega\, u_{\mathrm{PN}}\ud\bomega 
\end{equation}
so that
\begin{equation}\label{eq:u2phi_J}
   \bu^{\veps,h}_0 
   = \frac{\phi^{\veps,h}}{\sqrt{4\pi}}
   \qquad \text{and} \quad
   \bu^{\veps,h}_1
   = {\sqrt{\frac{3}{4\pi}}}\veps \bP^{-1}\bJ^{\veps,h},
\end{equation}
where the matrix
	\begin{equation}\label{eq:def_P}
		\bP= \begin{bmatrix}
			0 & 0 & 1 \\
			1 & 0 & 0 \\
			0 & 1 & 0
		\end{bmatrix}
	\end{equation}
is needed because the indexing of $\bomega$ and $\bfm_1$ is not consistent.

\begin{theorem}\label{th:asym_lim}
	Let $\bu^{\veps,h}$ be the solution to the spherical harmonic DG method \eqref{eq:bilinear} of the scaled RTE with parameter $\veps>0$. Suppose the assumptions \eqref{as:rte_scale_1}-\eqref{as:rte_scale_2} hold.  
 Then as $\veps\rightarrow 0$,
	\begin{equation}
		\phi^{\veps,h}\rightarrow \phi^{0,h}\in C_0\cap \bV^{h,k_0}_0, \quad \bJ^{\veps,h}\rightarrow \bJ^{0,h}\in \bV^{h,k_1}_1,	
	\end{equation}
	where $\phi^{0,h}$ and $\bJ^{0,h}$ solve the following system: For all $v_0\in C_0\cap \bV^{h,k_0}_0$ and $\bv_1\in \bV^{h,k_1}_1$, 
	\begin{subequations}\label{eq:rte_diff_lim}
		\begin{align}
			\sum_{K\in\mathcal{T}^h}\left\{-(\bJ^{0,h}, \nabla v_0)_K + (\sigma_{\mathrm{a}} \phi^{0,h},v_0)_K\right\} &= 4\pi\sum_{K\in\mathcal{T}^h} (\vint{f},v_0)_K, \label{eq:rte_diff_lim_a}\\
			\sum_{K\in\mathcal{T}^h}\bigg\{\frac{1}{3} \left(\nabla \phi^{0,h}, \bv_1\right)_K
			+ \left(\sigma_{\mathrm{t}} \bJ^{0,h},\bv_1\right)_K \bigg\} &= 0. \label{eq:rte_diff_lim_b}
		\end{align}
	\end{subequations}
\end{theorem}

\begin{proof}
	The first step is to determine bounds for $\bu^{\veps,h}$.
	Setting $v=\bu^{\veps,h}$ in \eqref{eq:bilinear} with the bilinear form \eqref{eq:a_b} and the flux \eqref{eq:flux_Pn}, and employing the equality \eqref{eq:stab} yield the energy equation
	\begin{multline}\label{eq:rte_dg_energy}
		\sum_{K\in\mathcal{T}^h} \bigg\{\frac{1}{4}\big(|\bn^K|\cdot\bD\,[\![\bu^{\veps,h}]\!], [\![\bu^{\veps,h}]\!]\big)_{\partial K}\\
		+\veps (\sigma_{\mathrm{a}}\bu^{\veps,h}, \bu^{\veps,h})_K + \left(\left(\frac{\sigma_{\mathrm{t}}}{\veps}-\veps\sigma_{\mathrm{a}}\right) \bR\bu^{\veps,h}, \bu^{\veps,h}\right)_K\bigg\} = \veps \sum_{K\in\mathcal{T}^h}(\bbf,\bu^{\veps,h})_K.
	\end{multline} 
	Since the first and the third terms on the left-hand side of the above equation are non-negative and $\sigma_{\mathrm{a}}>\sigma_{\mathrm{a}}^{\mathrm{min}}$, it follows that  
	\begin{equation}
	\sigma_{\mathrm{a}}^{\mathrm{min}} \|\bu^{\veps,h}\|^2
  \le 
  \sum_{K\in\mathcal{T}^h}(\bbf,\bu^{\veps,h})_K \le \frac{1}{2}\sum_{K\in\mathcal{T}^h}(\frac{1}{\sigma_{\mathrm{a}}^{\mathrm{min}} }\bbf, \bbf)_K + \frac{1}{2}\sum_{K\in\mathcal{T}^h}(\sigma_{\mathrm{a}}^{\mathrm{min}} \bu^{\veps,h}, \bu^{\veps,h})_K.
	\end{equation}
	Therefore $\sigma_{\mathrm{a}}^{\mathrm{min}} \|\bu^{\veps,h}\|^2 \lesssim \|\bbf\|_{L^2(X)}$, 
	i.e., the sequence $\{\bu^{\veps,h}\}_{\veps>0}$ is uniformly bounded. Therefore it has a convergent sub-sequence, for the ease of presentation, still denoted by $\{\bu^{\veps,h}\}_{\veps>0}$ with limit $\bu^{0,h}$.
	Next, since 
	\begin{equation}
		\sigma_{\mathrm{s}}^{\mathrm{min}} \|\bR\bu^{\veps,h}\|^2_X\le\sum_{K\in\mathcal{T}^h}((\sigma_{\mathrm{t}} - \veps^2 \sigma_{\mathrm{a}}) \bR\bu^{\veps,h}, \bu^{\veps,h})_K\le \veps^2 \sum_{K\in\mathcal{T}^h}(\bbf,\bu^{\veps,h})_K,
	\end{equation}
	it follows $\|\bR\bu^{\veps,h}\|^2_X \to 0$ as $\veps\rightarrow 0$.
	Because $\bR=\diag(0,1,1,\cdots,1,1)$ and \eqref{eq:u2phi_J}, we conclude that $\bu^{0,h}_\ell=0$ for all $\ell\ge 1$ and that $\bJ^{\veps,h}$ is bounded and converges to a sub-sequential limit $\bJ^{0,h}$. 
    Finally \eqref{eq:rte_dg_energy} also implies that 
    \begin{equation}
       \sum_{K\in\mathcal{T}^h}(|n_i|\bD^{(i)}[\![\bu^{\veps,h}]\!], [\![\bu^{\veps,h}]\!])_{\partial K} \le 4 \veps \sum_{K\in\mathcal{T}^h}(\bbf,\bu^{\veps,h})_K.
    \end{equation}
	As a result, $[\![\bu^{0,h}]\!]=\lim_{\veps\rightarrow 0}[\![\bu^{\veps,h}]\!] = 0$. In particular, $u^{0,h}_0$ is continuous.
	
	The next step is to find equations for $\phi^{\veps,h}$ and $\bJ^{\veps,h}$ in the limit. 
    By setting
    \(\bv^h=\veps^{-1}v_0\mathbf{e}_0\) in \eqref{eq:bilinear}, where $v_0\in C_0\cap \bV^{h,k_0}_0$ and $\mathbf{e}_0 = \begin{bmatrix} 1 & 0 & 0 & \cdots & 0 \end{bmatrix}^{\tran} \in \mathbb{R}^L$, 
    \begin{multline} \label{eq:3.10}
		\sum_{K\in\mathcal{T}^h}\bigg\{\sum_{i=1}^3 \left( n_i\overrightarrow{\bA^{(i)}\bu^{\veps,h}}, \veps^{-1}v_0^-\mathbf{e}_0 \right)_{\partial K} - \sum_{i=1}^3\left(\bA^{(i)}\bu^{\veps,h}, \left(\veps^{-1}\partial_i v_0\right)\mathbf{e}_0\right)_K\\
		+ \frac{1}{\sqrt{4\pi}}  (\sigma_{\mathrm{a}}\phi^{\veps,h},v_0)_K  - \sqrt{4\pi} (\Vint{f},v_0)_K\bigg\}=0.
    \end{multline}
        From \eqref{eq:A_struc} and \eqref{eq:u2phi_J}, we have
	\begin{equation}
		\bA^{(i)}\bu^{\veps,h} \cdot \mathbf{e}_0 = \bA^{(i)}_{0,1} \bu^{\veps,h}_1 = {\sqrt{\frac{3}{4\pi}}}\veps \bA^{(i)}_{0,1} \bP^{-1}\bJ^{\veps,h}.  
	\end{equation}     
    Noting the definition of \(\bJ^{\veps,h}\) in \eqref{eq:def_phi_J} and \(\bP\) in \eqref{eq:def_P}, and  \(\bA^{(i)}_{0,1} = \int_{\mathbb{S}}\omega_i \bfm_1\bfm^{\tran}_0\ud\bomega = \frac{1}{\sqrt{3}}\begin{bmatrix}
			\delta_{i,2} & \delta_{i,3} & \delta_{i,1}
		\end{bmatrix}^{\tran}\), we deduce by direct computation that
	\begin{equation}
		\bA^{(i)}_{0,1} \bP^{-1} \bJ^{\veps,h} = \frac{1}{\sqrt{3}} \left(\bJ^{\veps,h}\right)^i,  
	\end{equation} 
    where \(\left(\bJ^{\veps,h}\right)^i\) indicates the \(i\)th component of the vector \(\bJ^{\veps,h}\). Therefore, 
	\begin{equation}
		 \sum_{i=1}^3\left(\bA^{(i)}\bu^{\veps,h}, \left(\veps^{-1}\partial_i  v_0\right)\mathbf{e}_0\right)_K = \frac{1}{\sqrt{4\pi}}\sum_{i=1}^3\left( \left(\bJ^{\veps,h}\right)^i, \partial_i  v_0\right)_K = \frac{1}{\sqrt{4\pi}}(\bJ^{\veps,h},\nabla v_0)_K.
	\end{equation}     
	Since $v_0$ is continuous, and therefore $\sum_{K\in\mathcal{T}^h}\sum_{i=1}^3 \left(\overrightarrow{\bA^{(i)}\bu^{\veps,h}}, n_i\veps^{-1}v_0^-\mathbf{e}_0\right)_{\partial K}=0$, the equation \eqref{eq:3.10} reduces to
	\begin{equation}
		\sum_{K\in\mathcal{T}^h}\left\{-(\bJ^{\veps,h},\nabla v_0)_K + (\sigma_{\mathrm{a}} \phi^{\veps,h},v_0)_K\right\} = 4\pi \sum_{K\in\mathcal{T}^h} (\Vint{f},v_0)_K.
	\end{equation}
    By using the fact that $\bJ^{\veps,h}\rightarrow\bJ^{0,h}$ and $\phi^{\veps,h}\rightarrow \phi^{0,h}$, respectively, we can now pass to the limit as $\veps\rightarrow 0$ to obtain
	\begin{equation} \label{added1}
		\sum_{K\in\mathcal{T}^h}\left\{-(\bJ^{0,h}, \nabla v_0)_K + (\sigma_{\mathrm{a}} \phi^{0,h},v_0)_K\right\} = 4\pi\sum_{K\in\mathcal{T}^h} (\Vint{f},v_0)_K.
	\end{equation}
	Next, setting $\bv^h=\begin{bmatrix} 0 & \bv_1^{\tran} & 0 & \cdots & 0 \end{bmatrix}^{\tran} \in \mathbb{R}^L$ in \eqref{eq:bilinear} with $\bv_1\in \bV_1^{h,k_1}$ yields 
	\begin{multline} \label{eq3.16}
		\sum_{K\in\mathcal{T}^h}\bigg\{\sum_{i=1}^3 \left( n_i\overrightarrow{\bA^{(i)}\bu^{\veps,h}}, \bv^{h-} \right)_{\partial K} - \sum_{i=1}^3\left(\bA^{(i)}\bu^{\veps,h},\partial_i \bv^h\right)_K\\
		+ \sqrt{\frac{3}{4\pi}} \veps^2 (\sigma_{\mathrm{a}}\bP^{-1}\bJ^{\veps,h},\bv_1)_K + \sqrt{\frac{3}{4\pi}}((\sigma_{\mathrm{t}} - \veps^2\sigma_{\mathrm{a}}) \bP^{-1} \bJ^{\veps,h},\bv_1)_K - \veps (\bbf,\bv_1)_K\bigg\}=0.
	\end{multline}
    Since 
 \begin{equation}
     \sum_{i=1}^{3} \bP \bA^{(i)}_{1,0}  \partial_i \phi^{\veps,h} =  \frac{1}{\sqrt{3}}\nabla \phi^{\veps,h} 
 \end{equation}
and \(\bA^{(i)}_{1,0} = \int_{\mathbb{S}}\omega_i \bfm_1\bfm^{\tran}_0\ud\bomega = \frac{1}{\sqrt{3}}\begin{bmatrix}
			\delta_{i,2} & \delta_{i,3} & \delta_{i,1}
		\end{bmatrix}^{\tran}\), 
	it follows that 
	\begin{align}
		\sum_{i=1}^3 &\left(\left( n_i\overrightarrow{\bA^{(i)}\bu^{\veps,h}}, \bv^{h-} \right)_{\partial K} - \left( \bA^{(i)}\bu^{\veps,h},\partial_i \bv^h\right)_K \right)  \nonumber\\
		&= \sum_{i=1}^3 \left(\left( n_i\overrightarrow{\bA^{(i)}_{1,0}\bu^{\veps,h}_0} + n_i\overrightarrow{\bA^{(i)}_{1,2}\bu^{\veps,h}_2}, \bv^-_1 \right)_{\partial K} - \left( \bA^{(i)}_{1,0}\bu^{\veps,h}_0 + \bA^{(i)}_{1,2}\bu^{\veps,h}_2,\partial_i \bv_1\right)_K \right) \nonumber\\
		&= \sum_{i=1}^3 \bigg(\left( \frac{1}{\sqrt{4\pi}}\bA^{(i)}_{1,0}\phi^{\veps,h}, n_i\bv^-_1 \right)_{\partial K} - \left( \frac{1}{\sqrt{4\pi}}\bA^{(i)}_{1,0}\phi^{\veps,h},\partial_i \bv_1\right)_K\bigg) \nonumber\\
		&\phantom{=} \quad + \sum_{i=1}^3 \bigg(\left( n_i\overrightarrow{\bA^{(i)}_{1,2}\bu^{\veps,h}_2}, \bv^-_1 \right)_{\partial K} - \left( \bA^{(i)}_{1,2}\bu^{\veps,h}_2,\partial_i \bv_1\right)_K \bigg) \nonumber\\
		&= \sqrt{\frac{1}{12\pi}} (\bP^{-1}\nabla \phi^{\veps,h}, \bv_1)_K + \sum_{i=1}^{3} \left(\left( n_i\overrightarrow{\bA^{(i)}_{1,2}\bu^{\veps,h}_2}, \bv^-_1 \right)_{\partial K} - \left( \bA^{(i)}_{1,2}\bu^{\veps,h}_2,\partial_i \bv_1\right)_K \right). \label{eq:3.18}
	\end{align}
	Plugging \eqref{eq:3.18} in \eqref{eq3.16}, letting $\veps\rightarrow 0$ and noting $\bu^{0,h}_2= 0$ yield the limit equation 
	 	\begin{align} \label{added2}
	 		\sum_{K\in\mathcal{T}^h}\bigg\{\frac{1}{3} (\nabla \phi^{0,h}, \bv_1)_K
	 		+ (\sigma_{\mathrm{t}} \bJ^{0,h},\bv_1)_K \bigg\} &= 0.
	 	\end{align}
Together \eqref{added1} and \eqref{added2} recover \eqref{eq:rte_diff_lim}, and the proof is complete.
\end{proof}
Theorem \ref{th:asym_lim} states that $\phi^{\veps,h}$ converges to the function 
$\phi^{0,h}\in C_0\cap \bV_0^{h,k_0}$, which solves the system \eqref{eq:rte_diff_lim}. 
This system is a mixed discretization of the diffusion equation \eqref{eq:diffusion} with a continuous primal variable. 
An obvious conclusion from the above theorem is that we require $k_0\ge 1$ to ensure that $C_0\cap \bV_0^{h,k_0}$ is not a trivial space only consisting of constant functions. Therefore, we have the following corollary.
\begin{corollary}
	 $k_0\ge 1$ is a necessary condition for the solution $\phi^{\veps,h}$ converges to the correct limit solution
\end{corollary}

\subsection{Diffusion limit}\label{sec3.2}
Let us further consider the limit equations \eqref{eq:rte_diff_lim} of the system \eqref{eq:bilinear}. Assuming constant cross-sections on each element, 
for the first and second moments $\phi^{0,h}$ and $\bJ^{0,h}$, let us focus on $\mathbb{Q}_1$ elements on rectangular meshes. Choose $\bW_0^{h}:=C_0\cap \bV_0^{h,1}$ (i.e., continuous and piecewise $\mathbb{Q}_1$) and $\bW^h_1:=\bV_1^{h,1}$ (piecewise $\mathbb{Q}_1$). If $v_0\in \bW_0^{h}$, then $\bv_1:=\nabla v_0\in \bW^h_1$. Substitution yields a continuous Galerkin method for the diffusion equation \eqref{eq:diffusion}:
\begin{equation}\label{eq:diff_fem}
	\sum_{K\in\mathcal{T}^h}\left\{\frac{1}{3\sigma_{\mathrm{t}}} \left(\nabla \phi^{0,h}, \nabla v_0\right)_K + \sigma_{\mathrm{a}}( \phi^{0,h},v_0)_K\right\} = 4\pi\sum_{K\in\mathcal{T}^h} (\Vint{f},v_0)_K.
\end{equation}
It can be shown \cite{C1978} that this method \eqref{eq:diff_fem} converges with the order of $\mathcal{O}(h)$ in the standard $H^1$ norm. This shows the AP property of the spherical harmonic DG method, i.e., the limit (as $\veps\to 0$) of the discretization is the discretization of the limit equation. The same conclusion can also be drawn when $\mathbb{P}_1$ elements are used on triangular meshes.  

 We could also use lower order elements for $\bJ^{0,h}$. If we keep $\bW^h_0=C_0\cap \bV^{h,1}_0$, but take $\bW^h_1=\bV^{h,0}_1$ (the piecewise constant space) instead. On triangular meshes, we can take $\bV^{h,0}_1$ to be the piecewise $\mathbb{P}_1$ space, hence, we still have $\bv_1:=\nabla v_0\in \bW^h_1$ and following the same procedure of substitution gives the same limit equation \eqref{eq:diff_fem}, which implies the AP property on triangular meshes under this setup of approximation spaces. 
In fact, in the case that $k_0=1$ and $k_1 = 0$, the system \eqref{eq:rte_diff_lim} is a mixed discretization of the steady-state $P_1$ equations (i.e. the diffusion equation in first-order form), using the method from \cite{ES2012}. 

On rectangular meshes, we take $\bV_0^{h,1}$ to be the piecewise $\mathbb{Q}_1$ space. Since $\nabla v_0$ belongs to the piecewise $\mathbb{P}_1$ space and is no longer an element of $\bW^h_1=\bV^{h,0}_1$, we choose $\bv_1 := \overline{\nabla v_0}\in \bW^h_1$ in \eqref{eq:rte_diff_lim_b}, where $\overline{\nabla v_0}=(\nabla v_0,1)_K$ is the cell average.
Since $\bJ^{0,h}\in \bV^1_{h,0}$, we can obtain
\begin{equation}
	\left(\bJ^{0,h},\nabla v_0\right)_K=\bJ^{0,h}\left(1,\nabla v_0\right)_K=\left(\bJ^{0,h},\overline{\nabla v_0}\right)_K.
\end{equation}
Substituting these again into \eqref{eq:rte_diff_lim_a} to obtain
\begin{equation}\label{eq:diff_fem_avg}
	\sum_{K\in\mathcal{T}^h}\left\{\frac{1}{3\sigma_{\mathrm{t}}} \left(\nabla \phi^{0,h}, \overline{\nabla v_0}\right)_K + \sigma_{\mathrm{a}}( \phi^{0,h},v_0)_K\right\} = 4\pi\sum_{K\in\mathcal{T}^h} (\Vint{f},v_0)_K,
\end{equation}
or equivalently
\begin{equation}\label{eq:diff_fem_avgnew}
	\sum_{K\in\mathcal{T}^h}\left\{\frac{1}{3\sigma_{\mathrm{t}}} \left(\overline{\nabla \phi^{0,h}}, \overline{\nabla v_0}\right)_K + \sigma_{\mathrm{a}}( \phi^{0,h},v_0)_K\right\} = 4\pi\sum_{K\in\mathcal{T}^h} (\Vint{f},v_0)_K.
\end{equation}
For the one-dimensional model, it is easy to observe that \cref{eq:diff_fem_avg} is equivalent to \cref{eq:diff_fem}. For the multi-dimensional case, we can show that this method also converges with the order of $\mathcal{O}(h)$ in the standard $H^1$ norm. The proof follows from the standard error estimate in $H^1$ norm for \cref{eq:diff_fem}, combined with the fact that $\| \nabla \phi^{0,h} - \overline{\nabla \phi^{0,h}} \| = \mathcal{O}(h)$, where the bar denotes the cell-wise average.  The detailed steps are skipped here to save space. Therefore, the AP property of the spherical harmonic DG method can be achieved when lower order elements are used for $\bJ^{0,h}$, and this holds for both rectangular and triangular meshes. 

\subsection{A discontinuous Galerkin method with heterogeneous polynomial spaces for different moments}\label{sec:err_ana^heter}
In conventional spherical harmonics DG methods, it is standard that $k_i\equiv k$ for all $0\le i< N$. Furthermore, in order to capture the diffusion limit, local $\mathbb{P}_1$ or $\mathbb{Q}_1$ approximations for triangular or rectangular elements have to be employed, i.e., $k\ge 1$. 
However, a key observation from the asymptotic analysis for the DG method in \cref{th:asym_lim} and the discussion in Section \ref{sec3.2} indicates that, in order for numerical solutions of the DG method to have the correct asymptotic limit, the requirement $k\ge 1$ can be placed on the spaces $\bV_0^{h}$ only. This means that non-constant elements have to be employed only for the space $\bV^{h,k_0}_0$, i.e., $k_0\ge 1$, while a trivial space only consisting of constant $\mathbb{P}_0$ elements for simplices (and $\mathbb{P}_0$ or $\mathbb{P}_1$ for cuboids) can still be used for higher moments. 
This fact can be leveraged to save memory, hence the computational cost, while maintaining the property of AP and uniform convergence. For example, we could use $\mathbb{P}_1$ or $\mathbb{Q}_1$ elements for $\bu_0$ only, and $\mathbb{P}_0$ elements for all higher moments. Therefore, the improved DG method with heterogeneous polynomial spaces can be cast as follows. 
\begin{prob}\label{prob:SH_DG^heter}
	For $k_0\ge 1$, find $\bu^h\in \bV^{h,k_0}_0\times \bV^{h,0}_1\times \cdots \times \bV^{h,0}_N$ such that
	\begin{equation}\label{eq:bilinear^heter}
		\mathfrak{a}^h(\bu^h,\bv^h) = \mathfrak{f}(\bv^h) \quad \forall \bv^h\in \bV^{h,k_0}_0\times \bV^{h,0}_1\times \cdots \times \bV^{h,0}_N.
	\end{equation}
\end{prob}

\begin{remark}
For the proposed DG methods with heterogeneous polynomial spaces, high accuracy can be expected in the case of highly diffusive regions ($\veps\ll 1$). However, when $\veps = \mathcal{O}(1)$, the hybrid method will lose the high order accuracy due to the $\mathbb{P}_0$ elements employed. In fact, for simplicial elements, a uniform convergence (with respect to $\veps$) of order $\mathcal{O}(h^{1/2})$ can be proved when $\mathbb{P}_1$ and $\mathbb{P}_0$ elements are employed for $\bu_0$ and higher moments, respectively. The proof can be studied following the approach in \cite{SHX2023} and is neglected here to save space.
\end{remark}

\section{A hybrid discontinuous Galerkin/finite volume method}\label{sec:err_ana_hybrid}
The proposed DG method with heterogeneous polynomial spaces in the previous section can save computational costs but might lose convergence order when $\veps = \mathcal{O}(1)$. 
To alleviate such a drawback, we proposed a hybrid method, 
using reconstruction techniques from FV methods for all the 
$\mathbb{P}_0$ elements approximating higher moments. 

The proposed hybrid methods are a combination of DG methods and FV methods. 
For the unknown variables, the first moment of $\bu^h$ is approximated by a piecewise polynomial from the DG solution space, and all the other moments of $\bu^h$ are approximated
by a piecewise constant in the FV fashion. 
At each cell interface, we can compute the cell boundary values of $\bu^h$ via either polynomial evaluation (for the DG component) or reconstruction procedure (for the FV component). 
Numerical fluxes at the cell interface are then computed based on these values.
Below, we take the 1-D slab geometry model \eqref{eq:1d} as an example to present the approach.
The derivation for the 2-D model or fully RTE follows the same approach, and we provide some details and implementations of the algorithm for the 2-D model in \Cref{sec:appx}.

In 1-D slab geometry settings, the RTE model reduces to \eqref{eq:1d}, and the corresponding $P_N$ equation takes the form 
\begin{equation}\label{added3}
	\bB\partial_z\bu + \veps\sigma_{\mathrm{a}}\bu + \left(\frac{\sigma_{\mathrm{t}}}{\veps}-\veps\sigma_{\mathrm{a}}\right)\bR\bu = \veps \bbf.
\end{equation}
See \Cref{sec:appendix_PN_1d} for more details. 
We take the domain $I=(0,1)\subset \mathbb{R}$.
Let $\mathcal{T}^h$ be a regular family of meshes of $I$, i.e., $I:=\cup_{i=1}^M I_i$ with $I_i=(z_{i-1/2},z_{i+1/2})$. 
Denote $z_i=(z_{i-\frac{1}{2}}+z_{i+\frac{1}{2}})/2$, $h_i=z_{i+1/2}-z_{i-1/2}$ and $h:=\max_{1\le i\le M} h_i$. Let $\mathcal{E}^h$ be the set of all interior points of $\mathcal{T}$. 
The value of $\bu^h$ at $z=i+1/2$ is denoted by $(\bu^h)_{i+1/2}$ or just $\bu^h_{i+1/2}$. 

Let us set $\bu^h_0\in \bV^{h,1}_0$. Then for each cell $I_i$,
\begin{equation}
\label{eq:u0_linear}
\bu^h_{0}|_{I_i} = \overline{\bu}_{0,i} + \widetilde{\bu}_{0,i} \xi_i \in \mathbb{P}_1,
\quad \text{where $\xi_i=\frac{2(z-z_i)}{h_i}\in [-1,1]$}. 
\end{equation}
For higher moments, let $\bu^h_{\ell}\in \bV^{h,0}_\ell$. Then for each cell $I_i$,
\begin{equation}
\bu^h_{\ell}|_{I_i} = \overline{\bu}_{\ell,i}  \in [\mathbb{P}_0]^{2\ell+1}, \quad \ell=1,\cdots,N.
\end{equation}

Integration of \eqref{added3} over $I_i$ gives
\begin{equation}
\frac{1}{h_i}\left(\overrightarrow{\bB\bu}_{i+\frac{1}{2}} -\overrightarrow{\bB\bu}_{i-\frac{1}{2}}\right) + \veps\sigma_{\mathrm{a}}\overline{\bu}_i + \left(\frac{\sigma_{\mathrm{t}}}{\veps}-\veps\sigma_{\mathrm{a}}\right)\bR\overline{\bu}_i = \veps \overline{\bbf}_i.
\end{equation}
Here $\overrightarrow{\bB\bu}_{i+\frac{1}{2}}$ and $-\overrightarrow{\bB\bu}_{i-\frac{1}{2}}$ are the numerical fluxes and $\overline{\bbf}_i={\frac{1}{h_i}}\int_{I_i}\bbf\ud z$. 
Setting $\overrightarrow{\bB\bu}_{i+\frac{1}{2}}= \bB\avg{\bu}_{i+\frac{1}{2}} -\frac{1}{2}\bD[\![\bu]\!]_{i+\frac{1}{2}}$, where \(\bD =|\bB| \) is a positive definite matrix, gives
\begin{multline}
\label{eq:conservation_law_slab}
	\frac{1}{h_i}\Big(\bB\avg{\bu}_{i+\frac{1}{2}} -\frac{1}{2}\bD[\![\bu]\!]_{i+\frac{1}{2}}\Big) -\frac{1}{h_i}\Big(\bB\avg{\bu}_{i-\frac{1}{2}} + \frac{1}{2}\bD[\![\bu]\!]_{i-\frac{1}{2}}\Big) \\
	  + \veps\sigma_{\mathrm{a}}\overline{\bu}_i + \left(\frac{\sigma_{\mathrm{t}}}{\veps}-\veps\sigma_{\mathrm{a}}\right)\bR\overline{\bu}_i = \veps \overline{\bbf}_i,
\end{multline}

The edge values to define jumps and averages in \eqref{eq:conservation_law_slab} are determined using \eqref{eq:u0_linear} for $\ell=0$ and local reconstructions for $\ell >0$.  In the $\ell=0$ case, an equation for $\widetilde{\bu}_{0,i}$ is required.  To derive it, we multiply the $\ell=0$ component of \eqref{added3}  by $\xi_i$ and integrate over $I_i$.  
With the definition of the numerical fluxes above, the result is
\begin{multline}
	\frac{3}{h_i}\bigg(\Big(\bB\avg{\bu}_{i+\frac{1}{2}} -\frac{1}{2}\bD[\![\bu]\!]_{i+\frac{1}{2}}\Big)_0 +\Big(\bB\avg{\bu}_{i-\frac{1}{2}} + \frac{1}{2}\bD[\![\bu]\!]_{i-\frac{1}{2}}\Big)_0\bigg)\\
	-\frac{6}{h_i} (\bB\overline{\bu}_i)_0 + \veps\sigma_{\mathrm{a}}\widetilde{\bu}_{0,i}+ \left(\frac{\sigma_{\mathrm{t}}}{\veps}-\veps\sigma_{\mathrm{a}}\right) \bR\widetilde{\bu}_{0,i}= \veps \widetilde{\bbf}_{0,i},
\end{multline}
where, for notational convenience, we often use an outer subscript to extract the $\ell=0$ component of a vector, e.g. $(\bB\overline{\bu}_i)_0 = (\bB\overline{\bu})_{0,i}$.  

For $\ell>0$, we only know the cell average value of $\bu_\ell^h$. To recover uniform second-order accuracy, we employ Fromm's method to approximate slopes with center differences to find edge values and then apply upwind fluxes at cell interfaces. Let us assume a uniform mesh size $h$. Then 
\begin{equation}
	\bu^h_\ell\big|_{I_i} = \overline{\bu}_{\ell,i} + \frac{\overline{\bu}_{\ell,i+1} - \overline{\bu}_{\ell,i-1}}{2h}(z-z_i), \quad {\ell>0},
\end{equation}
so that {for \(\ell>0\), }
\begin{equation}
	\bu^{h-}_\ell\big|_{z_{i+\frac{1}{2}}} = \overline{\bu}_{\ell,i}+ \frac{\overline{\bu}_{\ell,i+1} - \overline{\bu}_{\ell,i-1}}{4}, \quad
	\bu^{h+}_\ell\big|_{z_{i+\frac{1}{2}}} = \overline{\bu}_{\ell,i+1} - \frac{\overline{\bu}_{\ell,i+2} - \overline{\bu}_{\ell,i}}{4}.
\end{equation}
Hence, for \(\ell>0\), 
\begin{subequations}\label{eq:fromm_1d}
	\begin{equation}
		\avg{\bu^h_\ell}_{i+\frac{1}{2}} = \frac{-\overline{\bu}_{\ell,i+2}+5\overline{\bu}_{\ell,i+1}+5\overline{\bu}_{\ell,i}-\overline{\bu}_{\ell,i-1}}{8},
	\end{equation}
	\begin{equation}
		[\![\bu^h_\ell]\!]_{i+\frac{1}{2}} = \frac{-\overline{\bu}_{\ell,i+2} + 3\overline{\bu}_{\ell,i+1} - 3\overline{\bu}_{\ell,i} + \overline{\bu}_{\ell,i-1}}{4},
	\end{equation}
	\begin{equation}
		[\![\bu^h_\ell]\!]_{i-\frac{1}{2}} = \frac{\overline{\bu}_{\ell,i+1} - 3\overline{\bu}_{\ell,i} + 3\overline{\bu}_{\ell,i-1} - \overline{\bu}_{\ell,i-2}}{4}.
	\end{equation}
\end{subequations}
Combining \eqref{eq:fromm_1d}, the mixed spherical harmonics DG method for the 1-D geometry setting can be cast as follows.
\begin{prob}
Find $\bu^h\in \bV^h := \bV^{h,1}_0\times\displaystyle\prod_{\ell=1}^N \bV^{h,0}_\ell$ such that 
	\begin{subequations}\label{eq:rte_dg_mixed_1d}
		\begin{align}
			\sum_{i=1}^M\Bigg\{\frac{1}{h}\Big(\bB\avg{\bu^h}_{i+\frac{1}{2}} -\frac{1}{2}\bD[\![\bu^h]\!]_{i+\frac{1}{2}}\Big)
			-\frac{1}{h}\Big(\bB\avg{\bu^h}_{i-\frac{1}{2}}
			+ \frac{1}{2}\bD[\![\bu^h]\!]_{i-\frac{1}{2}}\Big) \notag  \\
			+ \veps\sigma_{\mathrm{a}}\overline{\bu}^h_i + \left(\frac{\sigma_{\mathrm{t}}}{\veps}-\veps\sigma_{\mathrm{a}}\right) \bR\overline{\bu}^h_i \Bigg\}= \sum_{i=1}^M\veps \overline{\bbf}_{i},
		\end{align}
		\begin{align}
			\sum_{i=1}^M\Bigg\{\frac{3}{h}\Big(\bB\avg{\bu^h}_{i+\frac{1}{2}} -\frac{1}{2}\bD[\![\bu^h]\!]_{i+\frac{1}{2}}\Big)_0
			+\frac{3}{h}\Big(\bB\avg{\bu^h}_{i-\frac{1}{2}} + \frac{1}{2}\bD[\![\bu^h]\!]_{i-\frac{1}{2}}\Big)_0 \notag \\
			-\frac{6}{h} (\bB\overline{\bu}^h_i)_0 + \veps\sigma_{\mathrm{a}}(\widetilde{\bu}^h_i)_0 + \left(\frac{\sigma_{\mathrm{t}}}{\veps}-\veps\sigma_{\mathrm{a}}\right) \bR(\widetilde{\bu}^h_i)_0 \Bigg\}
			= \sum_{i=1}^M\veps \big(\widetilde{\bbf}_{i}\big)_0,
		\end{align}
	\end{subequations}
	where $\avg{\bu^h}_{i+\frac{1}{2}}$ and $[\![\bu^h]\!]_{i\pm\frac{1}{2}}$ are computed through \eqref{eq:fromm_1d}.
\end{prob}

In the proposed hybrid DG/FV method, solutions in each cell are still represented by at least linear polynomials for the $\ell=0$ moment and constants for higher $\ell>0$ moments. However, the values at each element interface for computing the upwind numerical fluxes of higher moments are determined by Fromm's method.  
Therefore, it can be considered as a further generalization of the DG method with heterogeneous polynomial spaces (\Cref{prob:SH_DG}) and be recast as follows.
\begin{prob}\label{prob:SH_DG^hyb}
	Find $\bu^{\veps,h}\in \bV^{h,k_0}_0\times \bV^{h,0}_1\times \cdots \times \bV^{h,0}_N$ such that for all $\bv^h\in \bV^{h,k_0}_0\times \bV^{h,0}_1\times \cdots \times \bV^{h,0}_N$, $k_0\ge 1$,
	\begin{equation}\label{eq:rte_dg}
		\sum_{K\in\mathcal{T}^h}\bigg\{ \left(\bn^K\cdot\overrightarrow{\bA\bu^{\veps,h}}, \bv^{h-}\right)_{\partial K} - \sum_{i=1}^3(\bA^{i}\bu^{\veps,h},\partial_i \bv^h)_K
        + \left(
        \bQ\bu^{\veps,h},\bv^h\right)_K - \veps (\bbf,\bv^h)_K\bigg\}=0,
	\end{equation}
 where the numerical trace $\bn^K\cdot\overrightarrow{\bA\bu^{\veps,h}}$ is defined by the formula \eqref{eq:flux_Pn}, with the boundary values of each element (used to evaluate jumps and averages) $\bu^{\veps,h}|_{\partial K}$ computed by Fromm's method. 
\end{prob}

\section{Numerical experiments}\label{sec:num_results}

In this section, we present some numerical tests to verify the theoretical results obtained in the previous section. To accommodate the most general settings, we consider the time-dependent RTE \eqref{eq:rte_scale_time_all} but only restrict ourselves to the spatial discretization errors for uniform convergence. 
The derivation of the \(P_N\) equations in the reduced geometry is described in \Cref{sec:appendix_PN}. 
The hybrid method introduced in \Cref{sec:err_ana_hybrid} consists of DG methods with $\mathbb{Q}_1$ element for the first moment and second-order FV methods for all other moments. 
The second-order Fromm's method is used as an example in this paper. The standard upwind flux is taken and the implicit BDF2 methods are used for the temporal discretization. Since BDF2 is a two-step method, we use the implicit backward Euler method to compute the solution at one time step. The time step $\Delta t$ is taken as \(\Delta t = 0.25 h\). 
To provide a comparison, we run the numerical examples using the hybrid method, the traditional DG method with $\mathbb{Q}_1$ element, and the second-order FV Fromm's method. 
To illustrate the low-memory property of the hybrid method, the numbers of unknowns of one cell assuming the \(P_N\) model are given in \Cref{tab:no_unkonwns} for each method in different dimensions.
\begin{table}[!htbp]
	\caption{Numbers of unknowns for different methods in different dimensions for \(P_N\) models}\label{tab:no_unkonwns}
	\begin{center}
		\begin{tabular}{c | c | c | c }
			\toprule
			 & DG with $\mathbb{Q}_1$ element & second-order FV & second-order hybrid  \\
			\midrule
			$1$D & $2(N+1)$ & \(N+1\) & \(N+2\) \\
                Ex.~in \S\ref{sect:5.1} & 4 & 2 & 3 \\        
                Ex.~in \S\ref{sect:5.2} & 8 & 4 & 5 \\
                \midrule
			$2$D & \(2(N+1)(N+2)\) & \((N+1)(N+2)/2\) & \((N+1)(N+2)/2+3\) \\
                Ex.~in \S\ref{sect:5.3} & \(40\) & \(10\) & \(13\) \\
                \midrule
			$3$D & \(8(N+1)^2\) & \((N+1)^2\) & \((N+1)^2+7\) \\
			\bottomrule
		\end{tabular}
	\end{center}
\end{table}

\subsection{One-dimensional \texorpdfstring{$P_1$}{P1} model}\label{sect:5.1}
When only the first two moments are considered, the spherical harmonic model becomes the ${P}_1$ model, taking the form 
\begin{subequations}\label{eq:p1-1d}
\begin{align}
\veps \rho_t+m_x&=0, \\
\veps m_t+\frac13 \rho_x&=-\frac{1}{\veps} m.    
\end{align}
\end{subequations}
where $\rho = u_0^0$ and $m = u_1^0$. The initial condition is given by 
\begin{equation}
	\rho(x,0)=\exp(-100(x-0.5)^2) , \qquad m(x,0)=0,
\end{equation}
on the domain $(0,\,1)$, and the boundary condition is periodic. 
Since the exact solution is not known explicitly for this case, we use the traditional DG method with $\mathbb{Q}_1$ element and $N=3200$ cells to compute a reference solution and treat this reference solution as the exact solution in computing the numerical errors.

In Figure \ref{fig.p1}, we plot the $L^2$ errors of the first variable $\rho$ for all three numerical methods (hybrid, DG, and FV methods) at time $t=0.05$ with $\veps=1$, $10^{-3}$, and $10^{-6}$, respectively. The number of cells $N$ is taken from $25$ to $800$. When $\veps=1$, all three numerical methods demonstrate a second-order convergence rate. 
When $\veps=10^{-6}$, the hybrid and DG methods have similar sizes of errors, while the FV method has a large error that does not decrease until the mesh size $h$ is small enough.

\begin{figure}
	\setcounter{subfigure}{0}
	\centering
	\subfloat[$\veps = 1$]{\includegraphics[width=0.33\textwidth,height=0.3\textwidth]{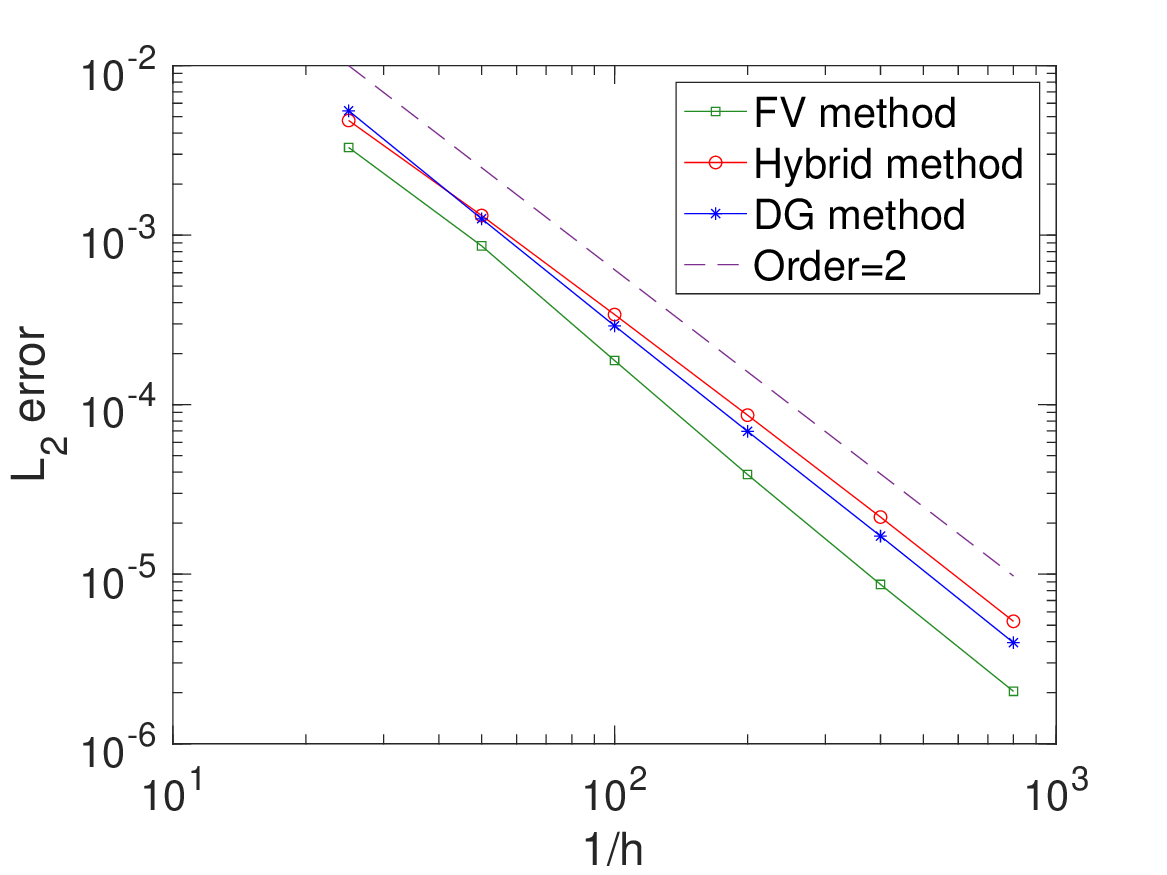}}
	\subfloat[$\veps = 10^{-3}$]{\includegraphics[width=0.33\textwidth,height=0.3\textwidth]{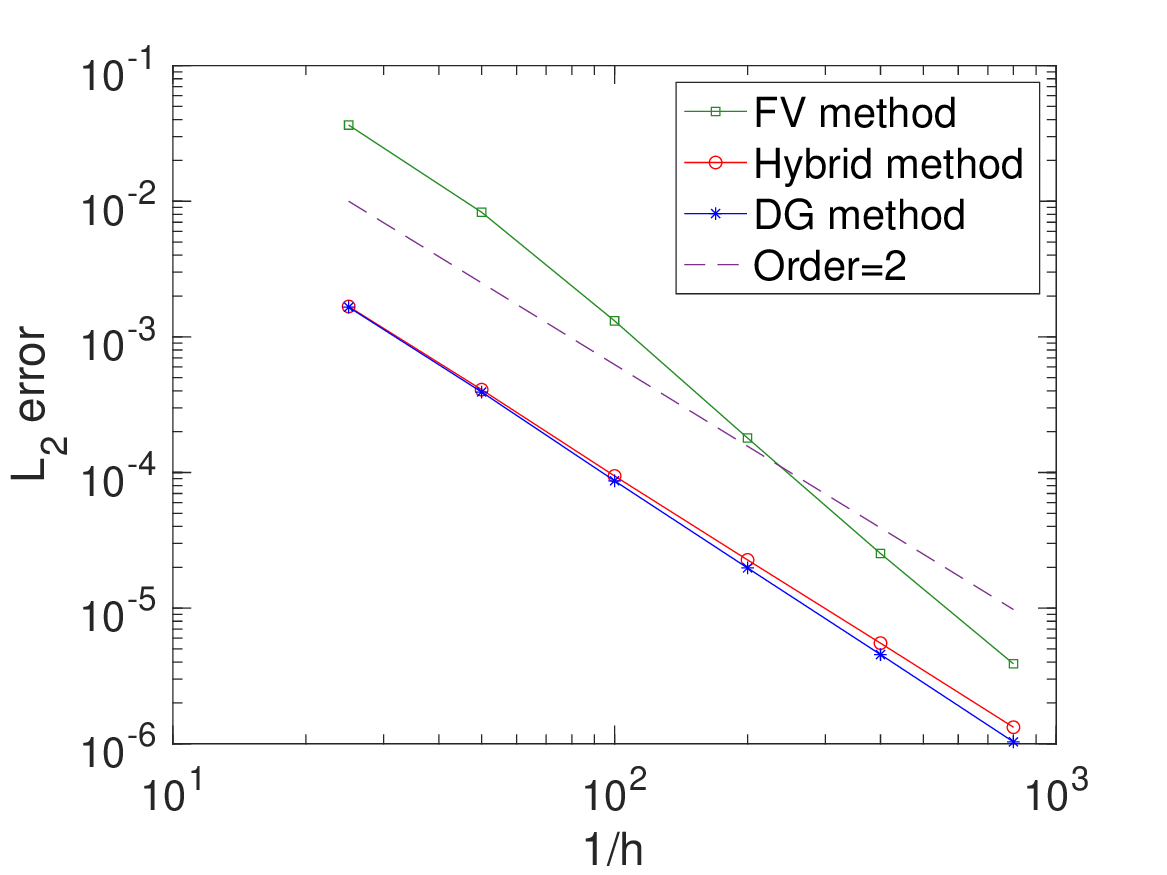}}
	\subfloat[$\veps = 10^{-6}$]{\includegraphics[width=0.33\textwidth,height=0.3\textwidth]{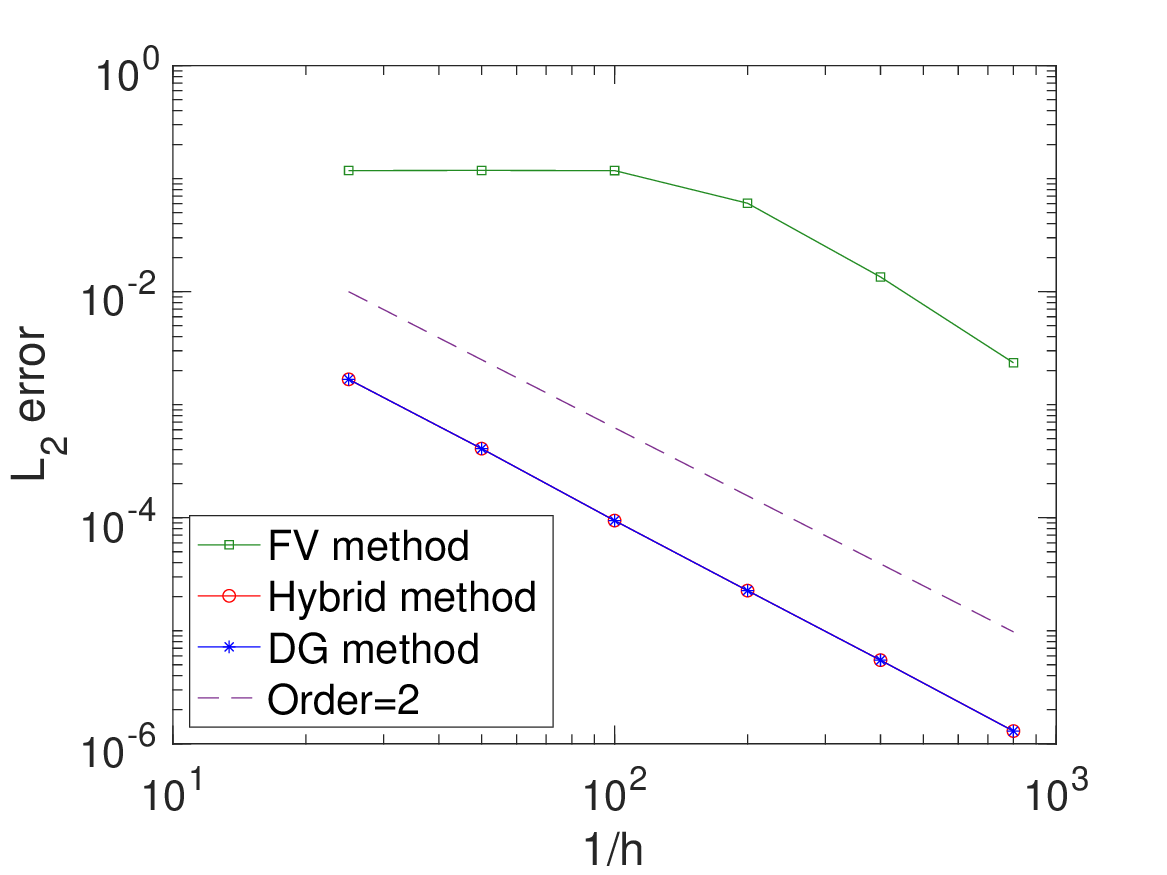}}
	\vspace{-0.15in} \caption{$L^2$ errors of $\rho^h$ for the one dimensional ${P}_1$ model at $t = 0.05$. }\label{fig.p1}
\end{figure}

\subsection{One-dimensional \texorpdfstring{$P_3$}{P3} model}\label{sect:5.2}
When only the first four moments are considered, the spherical harmonic model becomes the ${P}_3$ model, taking the form
\begin{subequations}
\label{eq:p3-1d}
\begin{align}
\veps \rho_t+m_x &=0, \\
\veps m_t+\frac13\rho_x+\frac23 w_x&=-\frac{1}{\veps} m, \\
\veps w_t+\frac25m_x+\frac35 q_x&=-\frac{1}{\veps} w, \\
\veps q_t+\frac37w_x&=-\frac{1}{\veps } q,    
\end{align}
\end{subequations}
where $\rho=u_0^0$, $m=u_1^0$, $w=u_2^0$, and $q=u_3^0$.
The initial condition is given by 
\begin{equation}
\rho(x,0)=\exp(-100(x-0.5)^2) , \qquad m(x,0)=0 , \qquad w(x,0)=0 , \qquad q(x,0)=0,
\end{equation}
on the domain $(0,\,1)$, 
 and the boundary condition is periodic. 
Since the exact solution is not known explicitly for this case, we use the traditional DG method with $\mathbb{Q}_1$ element and $N=3200$ cells to compute a reference solution and treat this reference solution as the exact solution in computing the numerical errors.

In Figure \ref{fig.p3}, we plot the $L^2$ errors of the first variable $\rho$ for all three numerical methods (hybrid, DG, and FV methods) at time $t=0.05$ with $\veps=1$, $10^{-3}$, and $10^{-6}$, respectively. The number of cells $N$ is taken from $25$ to $800$. When $\veps=1$, all three numerical methods demonstrate a second-order convergence rate. When $\veps=10^{-6}$, the hybrid and DG methods have similar sizes of errors, while the FV method has a large error that does not decrease until the mesh size $h$ is very small.  In \cite{lowrie2002methods}, it was shown that the formal consistency error of second-order FV methods typically scales like $\mathcal{O}(h^2 + \veps^{-1} h^3)$.  The behavior of the FV error in Figures \ref{fig.p1} and \ref{fig.p3} are consistent with this scaling.

\begin{figure}
	\setcounter{subfigure}{0}
	\centering
	\subfloat[$\veps = 1$]{\includegraphics[width=0.33\textwidth,height=0.3\textwidth]{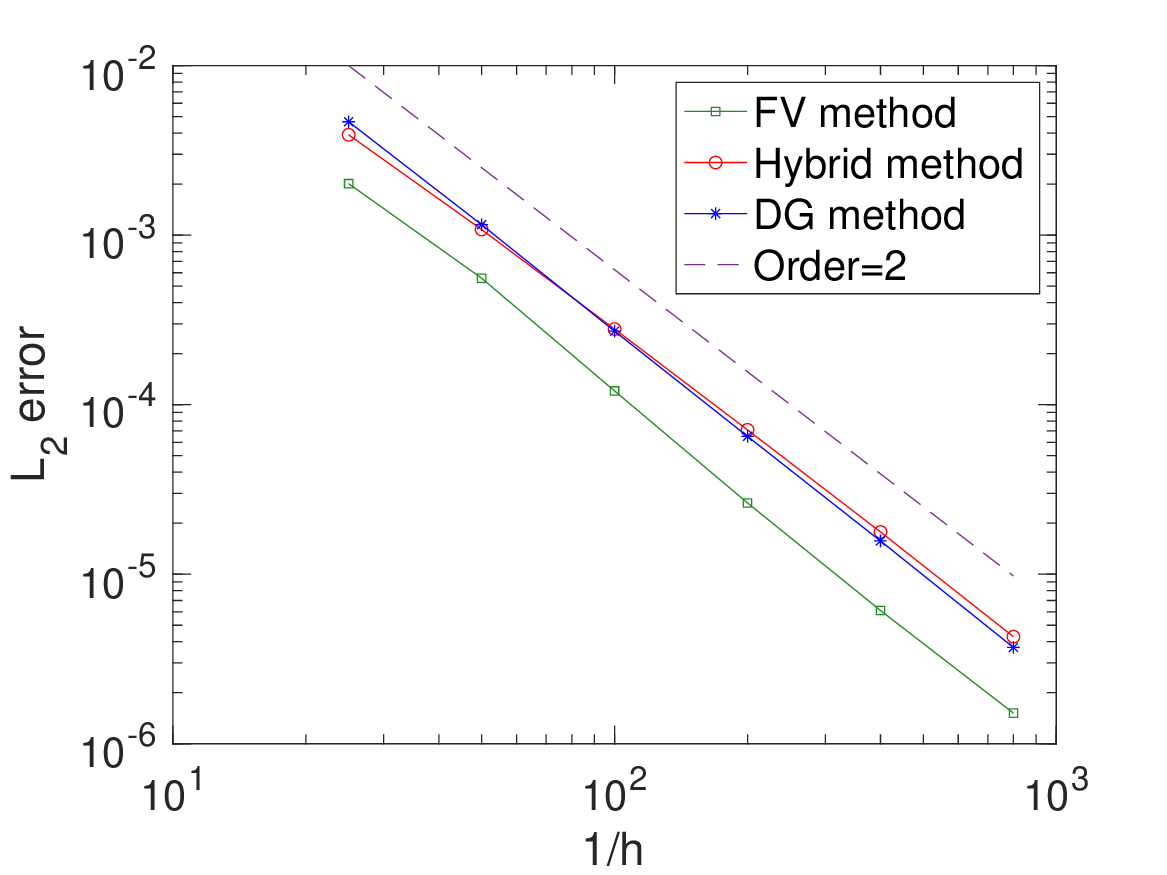}}
	\subfloat[$\veps = 10^{-3}$]{\includegraphics[width=0.33\textwidth,height=0.3\textwidth]{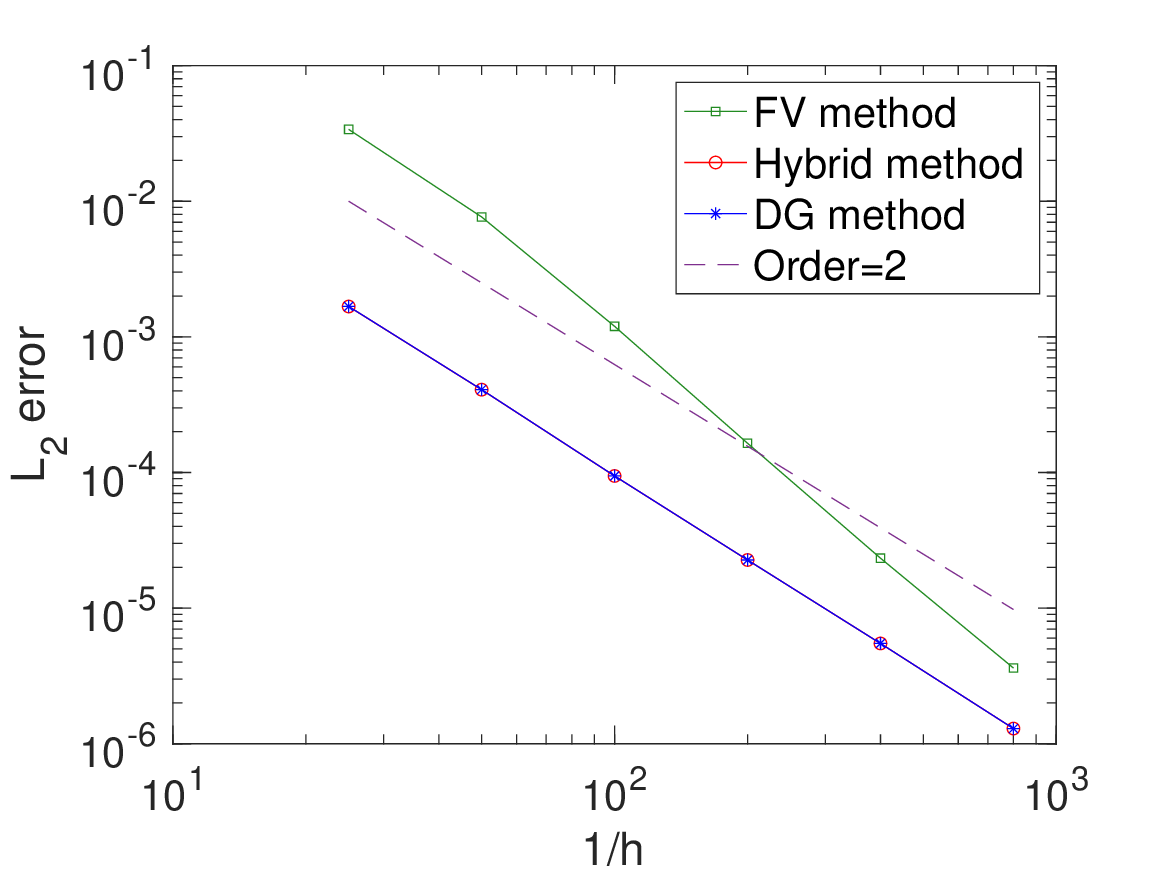}}
	\subfloat[$\veps = 10^{-6}$]{\includegraphics[width=0.33\textwidth,height=0.3\textwidth]{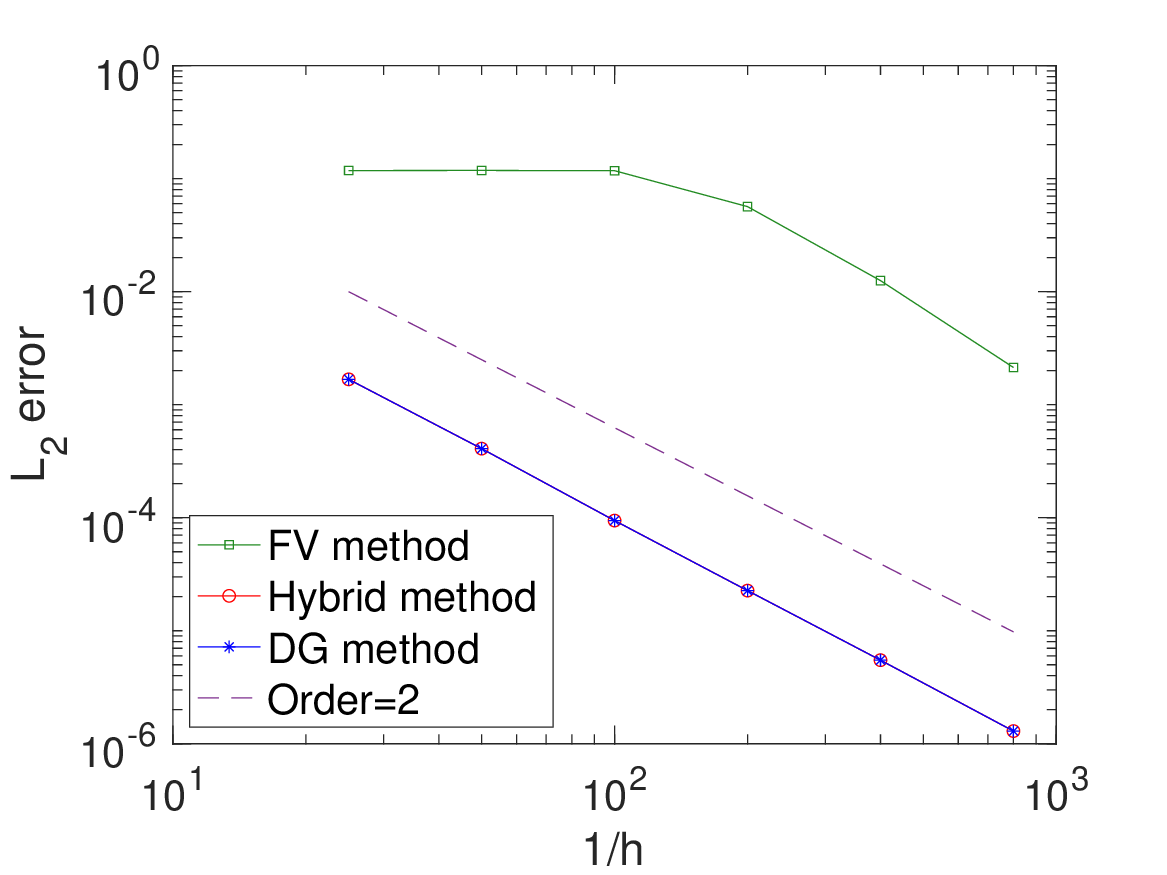}}
	\vspace{-0.15in} \caption{$L^2$ errors of $\rho^h$ for the one dimensional ${P}_3$ model at $t = 0.05$.} \label{fig.p3}
\end{figure}

\subsection{Two-dimensional \texorpdfstring{$P_3$}{P3} model}\label{sect:5.3}
In this subsection, we consider the two-dimensional ${P}_3$ model on the unit square domain $(0,\,1)\times(0,\,1)$.
The initial condition of the zeroth moment $\rho=u^0_0$ is given by  
\begin{equation}
	\rho(x,y,0)=1+\sin(2\pi x)\sin(2\pi y). 
\end{equation}
All higher-order moments are initially $zero$.
The periodic boundary condition is used for simplicity. 
Since the exact solution is not known explicitly for this case, we use the traditional DG method with $\mathbb{Q}_1$ element and $80\times 80$ cells to compute a reference solution for evaluating the numerical errors.

In Figure \ref{fig.2dp3}, we plot the $L^2$ errors of the first variable $\rho$ for all three numerical (hybrid, DG, and FV methods) at time $t=0.05$ with $\veps=1$, $10^{-3}$, and $10^{-6}$, respectively. The number of cells $N$ is taken from $5\times 5$ to $40\times 40$. When $\veps=1$, all three numerical methods demonstrate a second-order convergence rate. 
When $\veps=10^{-6}$, the hybrid and DG methods have similar sizes of error, while the FV method has a large error that does not decrease until the mesh size $h$ is small enough.

\begin{figure}
	\setcounter{subfigure}{0}
	\centering
	\subfloat[$\veps = 1$]{\includegraphics[width=0.33\textwidth,height=0.3\textwidth]{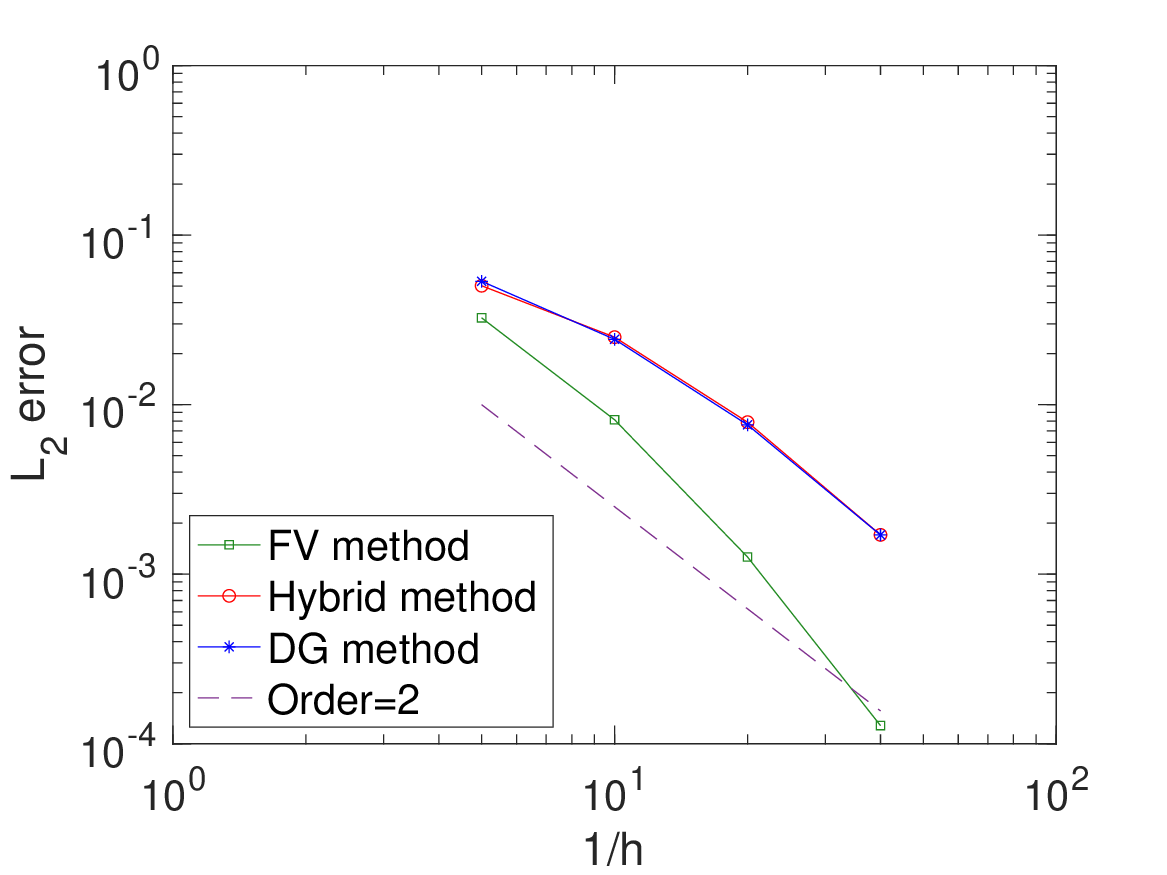}}
	\subfloat[$\veps = 10^{-3}$]{\includegraphics[width=0.33\textwidth,height=0.3\textwidth]{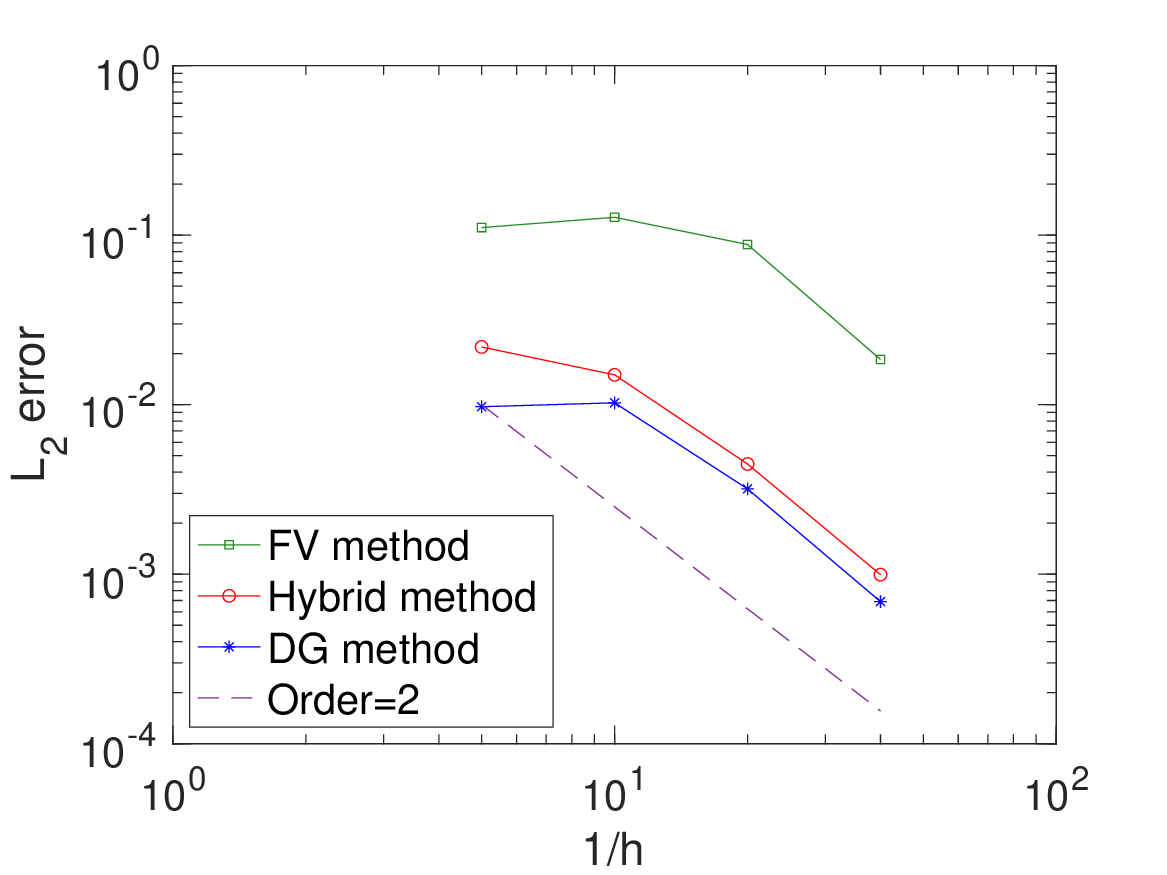}}
	\subfloat[$\veps = 10^{-6}$]{\includegraphics[width=0.33\textwidth,height=0.3\textwidth]{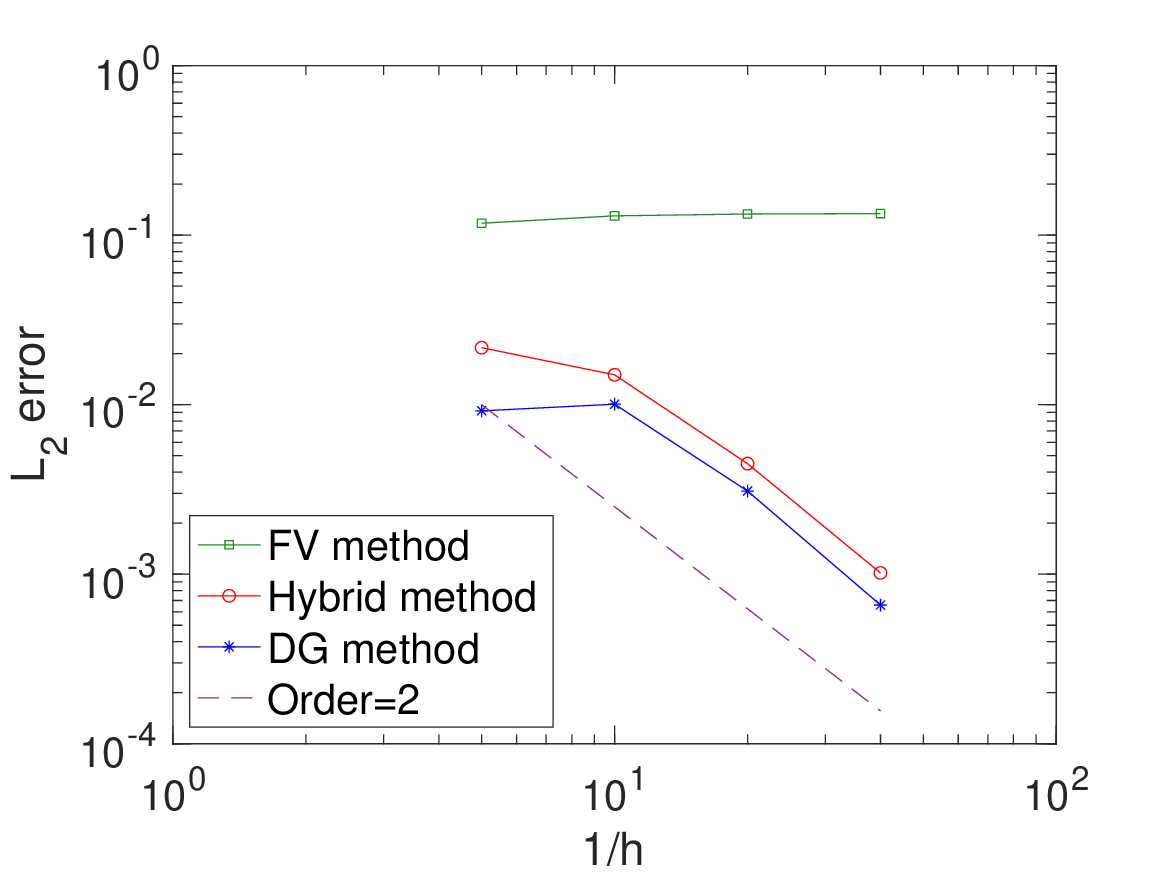}}
	\vspace{-0.15in} \caption{ $L^2$ errors of $\rho^h$ for the two dimensional $P_3$ model at $t = 0.05$.} \label{fig.2dp3}
\end{figure}

In Figure \ref{fig.2dp3.2}, we plot the numerical solutions of $\rho^h$ for all three methods at the final time $t=0.05$ with $\veps=10^{-6}$. 
For such a small $\veps$, the $P_3$ model is very close to the limit diffusion problem. As a comparison, we also solve the limit diffusion equation by a local DG method and plot its solution in Figure \ref{fig.2dp3.2}. We can observe that the hybrid and DG methods provide numerical solutions which are very similar to that of the limit diffusion equation, while the FV (non-AP) solution is almost flat because of the large amount of numerical diffusion.

\begin{figure}
	\setcounter{subfigure}{0}
	\centering
	\subfloat[DG]{\includegraphics[width=0.4\textwidth]{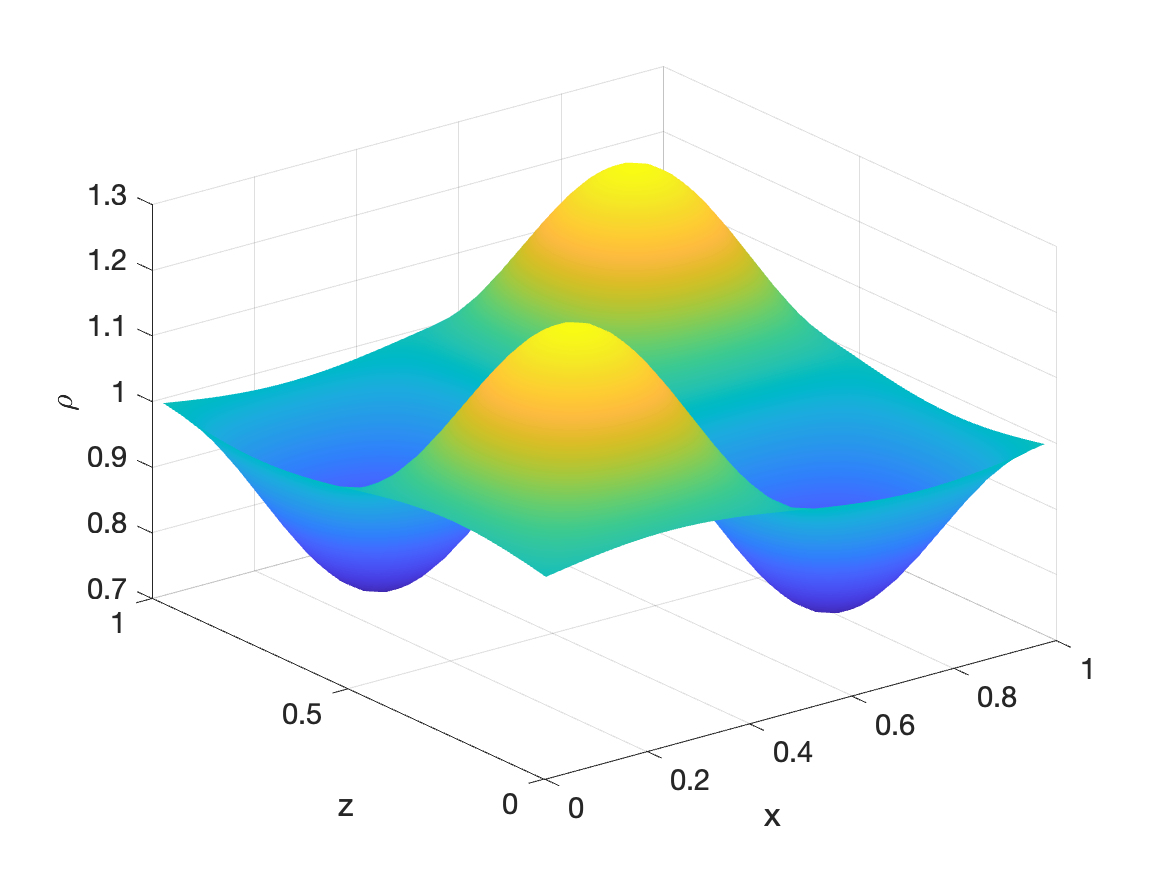}}
	\subfloat[Hybrid method]{\includegraphics[width=0.4\textwidth]{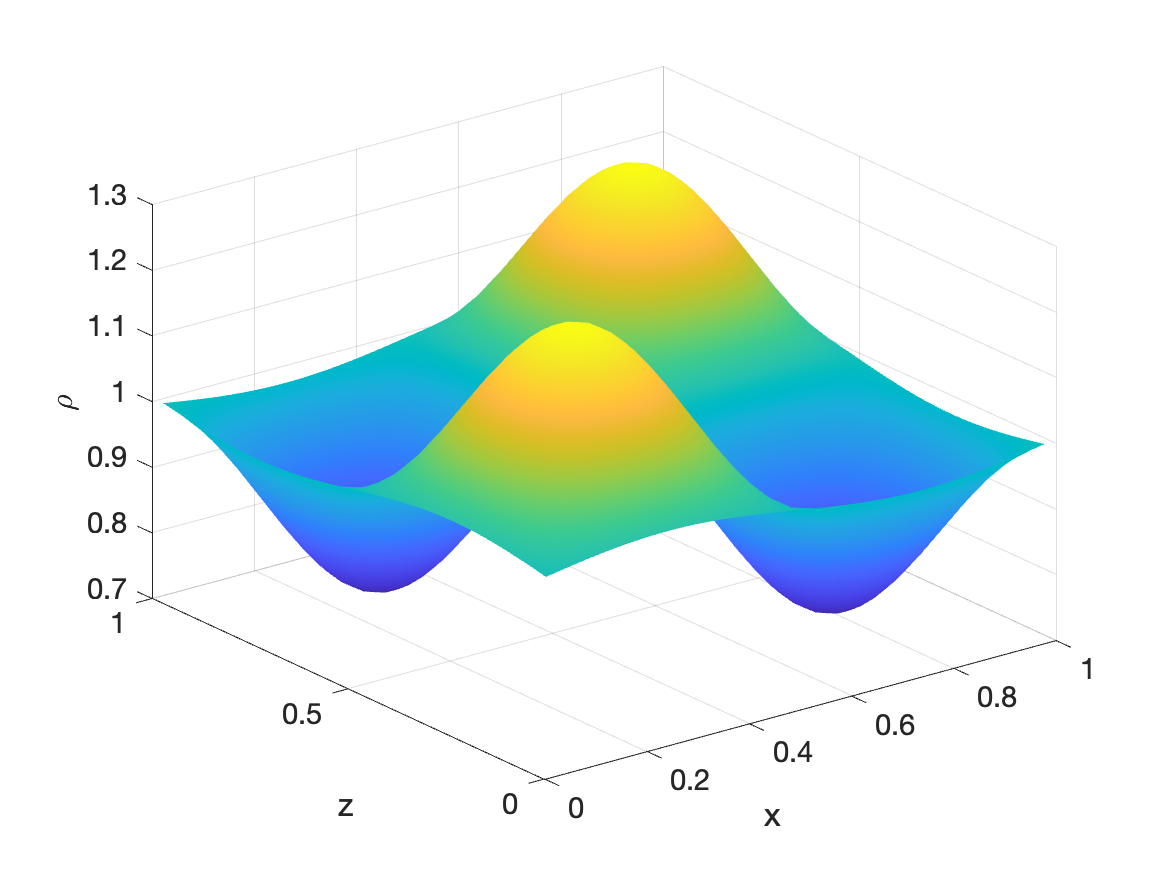}} \\
	\subfloat[Diffusion]{\includegraphics[width=0.4\textwidth]{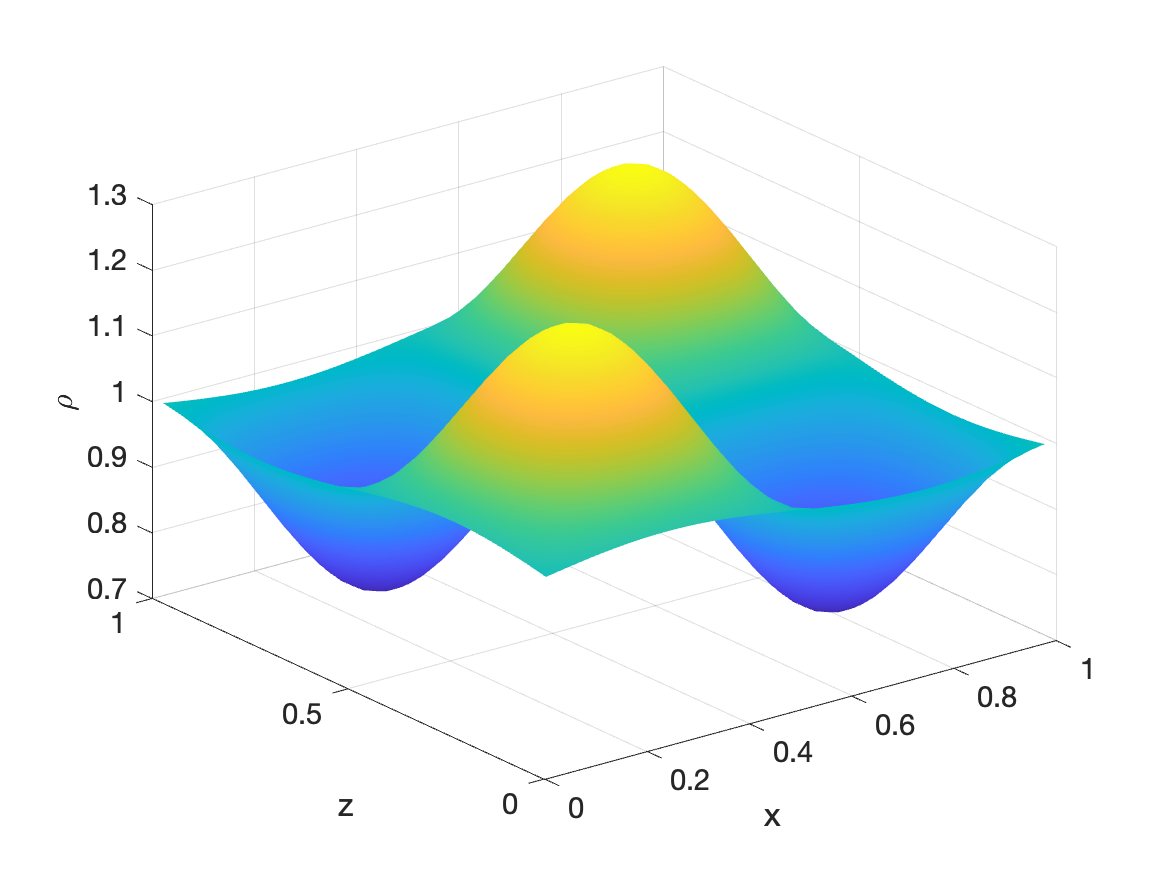}}
	\subfloat[FV]{\includegraphics[width=0.4\textwidth]{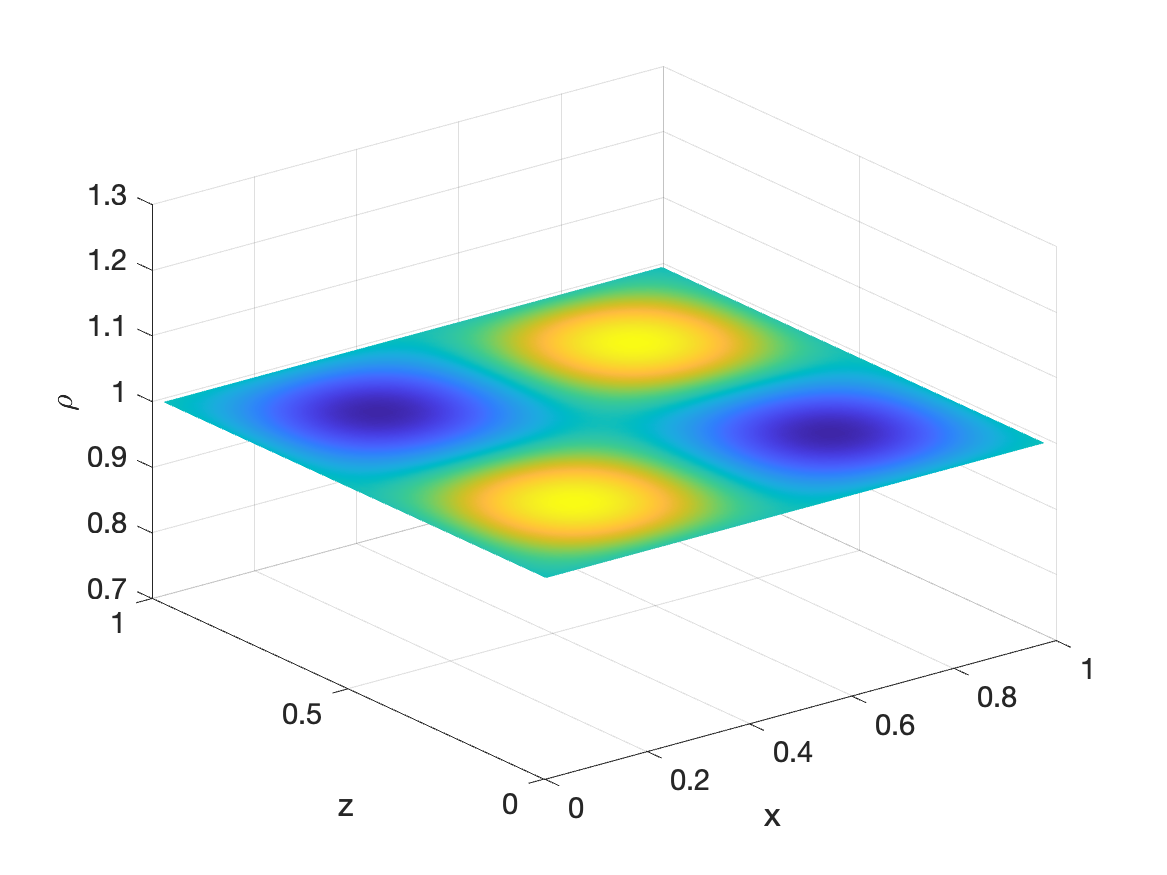}}
	\vspace{-0.15in} \caption{Numerical solution $\rho^h$ for the two dimensional $P_3$ model at $t = 0.05$.} \label{fig.2dp3.2}
\end{figure}

\section{Concluding remarks}\label{sec:conclusion}
In this paper, we have performed an asymptotic analysis on a spherical harmonics DG discretization for the radiative transport equation to show that uniform convergence with respect to the scaling parameter $\veps$ can be achieved using piece-wise linear (or higher order) elements only for the degree zero moment while using constant elements to approximate higher moments. Based on this observation, we propose a new DG method with heterogeneous polynomial spaces for different moments.  This results in a method that still captures the diffusion limit when $\veps \to 0$, but uses fewer degrees of freedom in physical space.  However, the accuracy of the method is reduced when $\veps \sim 1$.

To overcome the reduction of uniform accuracy in the heterogeneous DG method, we propose a hybrid discontinuous Galerkin/finite volume scheme in which FV reconstructions are used for the higher moments. 
The new method is asymptotic preserving, and it maintains the convergence rate of the conventional DG method, while still using fewer degrees of freedom.  Some numerical experiments, including the comparison with other techniques, are carried out to test the performance of the scheme and show the effectiveness of the method.


\section*{Acknowledgments}
The work of C. Hauck is partially supported by NSF Grant DMS-1913277 and by the Department of Energy Office of Advanced Scientific Computing Research and the National Science Foundation (1217170); 
ORNL is operated by UT-Battelle, LLC., for the U.S. Department of Energy under Contract DE-AC05-00OR22725. The United States Government retains and the publisher, by accepting the article for publication, acknowledges that the United States Government retains a non-exclusive, paid-up, irrevocable, world-wide license to publish or reproduce the published form of this manuscript, or allow others to do so, for the United States Government purposes. The Department of Energy will provide public access to these results of federally sponsored research in accordance with the DOE Public Access Plan (http://energy.gov/downloads/doe-public-access-plan);
The work of Y. Xing is partially supported by the NSF grant DMS-1753581 and DMS-2309590.


\appendix\label{sec:appendix}
\section{\texorpdfstring{\(P_N\)}{PN} equations for reduced geometries model}\label{sec:appendix_PN}
In this section, we connect the general $P_N$ equations \eqref{eq:rte_scale_Pn} to the reduced geometry equations in \Cref{sec:num_results}.  We focus on the steady-state equations, but the time-dependent versions simulated in \Cref{sec:num_results} can be derived in a similar fashion.

Recall the components of $\bomega$ in \eqref{eq:omega_components} expressed in terms of the polar angle $\theta$ and azimuthal angle $\varphi$.  The normalized, complex-valued spherical harmonics are
\begin{equation}
    Y_{\ell}^{\kappa}(\bomega) 
    = \sqrt{\frac{2\ell+1}{4 \pi}}\frac{(\ell-\kappa)!}{(\ell+\kappa)!} P_\ell^\kappa (\mu) e^{i \kappa \varphi},
\end{equation}
where $\mu = \cos \theta$ and the associated Legendre functions are given by
\begin{equation}
    P_\ell^\kappa(\mu) =\begin{cases}
         (-1)^\kappa (1-\mu^2)^{m/2} \frac{d^{(\kappa)}}{d \mu^{(\kappa)}} P_\ell(\mu),
        & \kappa \geq 0,\\
         (-1)^{|\kappa|} \frac{(\ell+\kappa)!}{(\ell-\kappa)!}  P_\ell^{|\kappa|},
        & \kappa < 0,
    \end{cases}
\end{equation}
and $P_\ell \colon [-1,1] \to \mathbb{R}$ is the degree $\ell$ Legendre polynomial, normalized so that $\int_{-1}^1 P_\ell P_{\ell'} d \mu = \frac{2}{2\ell+1} \delta_{\ell,\ell'}$.  The normalized, real-valued spherical harmonics are
\begin{equation}
    m_{\ell}^{\kappa}(\bomega) 
    =\begin{cases}
    \sqrt{2} (-1)^\kappa \operatorname{Im}[Y_\ell^{|\kappa|}] ,
    & \kappa < 0, \\
    Y_\ell^{0},
    & \kappa = 0, \\
    \sqrt{2} (-1)^\kappa \operatorname{Re}[Y_\ell^{|\kappa|}],
    & \kappa > 0.
\end{cases}
\end{equation}

\subsection{One-dimensional slab geometry}\label{sec:appendix_PN_1d}

Recall the slab geometry RTE in \eqref{eq:1d}.  The equation assumes that $u$ and $f$ are independent of $x$, $y$ and $\varphi$.  If $u_{\mathrm{PN}}$ is also assumed to be independent of $x$, $y$ and $\varphi$, then for any $\kappa \ne 0$,
\begin{equation}
    u_\ell^\kappa  
    = \int_{-1}^1 u_{\mathrm{PN}}(z,\mu)
    \left(\int_0^{2 \pi} m_\ell^\kappa(\bomega)  d \varphi \right) d \mu
    = 0.
\end{equation}
Thus the only non-zero moments of $u$ are with $\kappa = 0$, i.e., with respect to $Y_{\ell}^{0}(\bomega) = \sqrt{\frac{2\ell+1}{4 \pi}} P_\ell (\mu)$.
A common convention \cite[Appendix D]{LM1984} for slab geometry is to express  the $P_N$ equations in terms of 
\begin{equation}
    v_\ell = \frac{1}{2}\int_{-1}^1 u_{\mathrm{PN}} P_\ell(\mu)\ud \mu
    = \frac{1}{\sqrt{4 \pi}} \frac{1}{\sqrt{2 \ell+1}} u_\ell^0,
\end{equation}
so that 
\begin{equation}
    u_{\mathrm{PN}} = \sum_{\ell = 0}^N (2\ell+1) v_\ell P_\ell.
\end{equation}
Setting this expansion into \eqref{eq:1d} and integrating the result against $\frac{1}{2} P_\ell$ over $\mu \in [-1,1]$ give the following equation for $\bv = \begin{bmatrix} v_0 & \cdots & v_N \end{bmatrix}^\tran$:
\begin{equation}\label{eq:PN_1D}
     \bB \partial_z \bv 
    + \veps\sigma_{\mathrm{a}}\bv 
    + \left(\frac{\sigma_{\mathrm{t}}}{\veps}-\veps\sigma_{\mathrm{a}}\right) {\bR}\bv 
    = \veps {\bbf},
\end{equation}
where $f_\ell = \frac{1}{2}\int_{-1}^1 f P_\ell(\mu)\ud \mu$, $R_\ell = \delta_{\ell,0}$, and for $0 \leq \ell,\ell' \leq N$, the matrix $\bB$ has elements
\begin{equation}
    \bB_{\ell,\ell'} 
    =  \frac{2 \ell' +1}{2} \int_{-1}^1 P_\ell(\mu) P_{\ell'}(\mu) d \mu
    = \frac{\ell+1}{2 \ell +1} \delta_{\ell+1,\ell'} +
        \frac{\ell}{2 \ell +1} \delta_{\ell-1,\ell'}.
\end{equation}
The matrix $\bB$ is the flux matrix that appears in \eqref{eq:p1-1d} when $N=1$ and in \eqref{eq:p3-1d} when $N=3$. 
\subsection{2-D plane-parallel model}
We summarize here the approach used in \cite{brunner2005two} to derive the \(P_N\) equation for the 2-D plane-parallel model.
Let us assume that $\sigma_{\mathrm{t}}$, $\sigma_{\mathrm{a}}$, and $f$ do not depend on $y$, are the functions only of $x$, $z$, $\theta$, and $\varphi$, and are even in $\varphi - \pi$. Furthermore, the solution \(u\) is independent of $y$. Then obviously, the solution is a function only of $x$, $z$, $\theta$, and $\varphi$. Set $\mu=\cos\theta$, and the time-dependent scaled radiative transfer equation can be written in the following form
\begin{equation}\label{eq:rte_2D_y}
 \sqrt{1-\mu^2}\cos\varphi\frac{\partial u}{\partial x} + \mu\frac{\partial u}{\partial z} + \frac{\sigma_{\mathrm{t}}}{\veps}u
		=\left(\frac{\sigma_{\mathrm{t}}}{\veps}-\veps\sigma_{\mathrm{a}}\right)\Vint{u}+\veps f,
\end{equation}
where $\mu=\cos\theta\in[0,1]$, $\varphi\in [0,2\pi)$, $u=u(x,z,\mu,\varphi)$, and $f=f(x,z,\mu,\varphi)$. 
It is easy to show that if \(u(x,z,\varphi,\mu)\) is a solution of \eqref{eq:rte_2D_y}, then so too is  \(u(x,z,2 \pi -\varphi,\mu)\).
Hence \(u\) is also an even function of \(\varphi-\cpi\). 

In lieu of real-valued spherical harmonics, we employ the complex-valued normalized spherical harmonics \(Y_\ell^\kappa\) of degree \(\ell\) and order \(\kappa\) for \(-\ell\le\kappa\le\ell\) and \(0\le\ell\le N\). 
We have the following recursive relations
\begin{subequations}\label{eq:recur_comp}
    \begin{align}
	\mu Y_\ell^k&=A_\ell^k Y_{\ell+1}^k + B_\ell^k Y_{\ell-1}^k, \\
	\sqrt{1-\mu^2}\cos\varphi Y_\ell^k &= \frac{1}{2}\left(-C_\ell^k Y_{\ell+1}^{k+1} + D_\ell^k Y_{\ell-1}^{k+1} + E_\ell^k Y_{\ell+1}^{k-1} -  F_\ell^k Y_{\ell-1}^{k-1}\right),
\end{align}
\end{subequations}
with
\begin{alignat}{2}
	A_\ell^k&=\sqrt{\frac{(\ell-k+1)(\ell+k+1)}{(2\ell+3)(2\ell+1)}}, &\qquad B_\ell^k&=\sqrt{\frac{(\ell-k)(\ell+k)}{(2\ell+1)(2\ell-1)}}, \\
	C_\ell^k&=\sqrt{\frac{(\ell+k+1)(\ell+k+2)}{(2\ell+3)(2\ell+1)}}, & D_\ell^k&=\sqrt{\frac{(\ell-k)(\ell-k-1)}{(2\ell+1)(2\ell-1)}}, \\
	E_\ell^k&=\sqrt{\frac{(\ell-k+1)(\ell-k+2)}{(2\ell+3)(2\ell+1)}}, & F_\ell^k&=\sqrt{\frac{(\ell+k)(\ell+k-1)}{(2\ell+1)(2\ell-1)}}. 
\end{alignat}

Similar to the procedure of obtaining the one-dimensional \(P_N\) equation \eqref{eq:PN_1D}, we multiply \eqref{eq:rte_2D_y} by the complex-valued normalized spherical harmonics $Y_\ell^\kappa$ and employ the recursive relations \eqref{eq:recur_comp} to obtain
\begin{multline}\label{eq:rte_scale_Pn_2D_y_scalar}
    \frac{1}{2}\partial_x\left(-C_{\ell-1}^{\kappa-1}u_{\ell-1}^{\kappa-1} + D_{\ell+1}^{\kappa-1}u_{\ell+1}^{\kappa-1} + E_{\ell-1}^{\kappa+1}u_{\ell-1}^{\kappa+1} - F_{\ell+1}^{\kappa+1}u_{\ell+1}^{\kappa+1}\right) \\ 
    +\partial_z\left(A_{\ell-1}^{\kappa}u_{\ell-1}^{\kappa} + B_{\ell+1}^{\kappa}u_{\ell+1}^{\kappa}\right) 
    + \veps\sigma_{\mathrm{a}}u_{\ell}^{\kappa}
    + \left(\frac{\sigma_{\mathrm{t}}}{\veps}-\veps\sigma_{\mathrm{a}}\right)(1-\delta_{\ell0}\delta_{\kappa0})u_{\ell}^{\kappa}=\veps f_{\ell}^{\kappa},
\end{multline}
for \(0\le \ell\le N\) and \(-\ell\le \kappa\le\ell\) with \(f_{\ell}^{\kappa}=4\cpi\Vint{fY_\ell^\kappa}\). 
Note that the real and imaginary parts of the above equation are decoupled. Furthermore, since \(f\) is an even function of \(\varphi-\cpi\), we have \(\Vint{f P_\ell^k(\mu)\sin(k(\varphi-\cpi))}=0\). Therefore, \(f_{\ell}^{\kappa}\) is real, and we only need to solve for the real parts of the equation \eqref{eq:rte_scale_Pn_2D_y_scalar}.

We denote \(\bu\) as a vector of spherical harmonics moments of \(u\) grouped in blocks of constant \(\kappa\)
\begin{equation}
    \bu=\begin{bmatrix}
		\left(\bu^{-N}\right)^\tran & \left(\bu^{-N+1}\right)^\tran & \cdots & \left(\bu^{-1}\right)^\tran & \left(\bu^{0}\right)^\tran & \left(\bu^{1}\right)^\tran & \cdots & \left(\bu^{N}\right)^\tran 
	\end{bmatrix}^\tran
\end{equation}
with \(\bu^\kappa=\begin{bmatrix}
		u_{|\kappa|}^{\kappa} & u_{|\kappa|+1}^{\kappa} & \cdots & u_{N}^{\kappa}
	\end{bmatrix}^\tran\). 
Then the matrix form of the 2D \(P_N\) equation \eqref{eq:rte_scale_Pn_2D_y_scalar} can be written as 
\begin{equation}\label{eq:rte_scale_Pn_2D_y_full}
 \bA_x\partial_x\bu + \bA_z\partial_z\bu + \veps\sigma_{\mathrm{a}}\bu + \left(\frac{\sigma_{\mathrm{t}}}{\veps}-\veps\sigma_{\mathrm{a}}\right) \bR\bu = \veps \bbf,
\end{equation}
which has the same format as \eqref{eq:rte_scale_Pn} if we set \(\bA_y=0\). 

The number of unknowns in \eqref{eq:rte_scale_Pn_2D_y_full} can be further reduced due to the special structure of matrix \(\bA_x\) and \(\bR\). 
First,  if $\bu$
is solution of \eqref{eq:rte_scale_Pn_2D_y_full}, then, due to the symmetry of the equations, 
 \begin{equation}\bu'=\begin{bmatrix}
		\left(\bu^{N}\right)^\tran & \cdots & \left(\bu^{1}\right)^\tran & \left(\bu^{0}\right)^\tran & \left(\bu^{-1}\right)^\tran & \cdots & \left(\bu^{-N}\right)^\tran 
	\end{bmatrix}^\tran\end{equation} 
is also a solution. Therefore, by the uniqueness
\begin{equation}
    \bu^{\kappa}=\bu^{-\kappa}, \quad \kappa=0,1,\cdots,N.
\end{equation}
In particular, since \(u_\ell^{-1}=u_\ell^1\), \(C_\ell^{-1}=E_\ell^1\), and \(D_\ell^{-1} = F_\ell^1\), we can eliminate the \(u_\ell^{-1}\) terms from \eqref{eq:rte_scale_Pn_2D_y_scalar} for the \(\kappa = 0\) case:
\begin{multline}\label{eq:rte_scale_Pn_2D_y_scalar_m=0}
    \partial_x\left(E_{\ell-1}^{1}u_{\ell-1}^{1} - F_{\ell+1}^{1}u_{\ell+1}^{1}\right) \\ 
    +\partial_z\left(A_{\ell-1}^{0}u_{\ell-1}^{0} + B_{\ell+1}^{0}u_{\ell+1}^{0}\right) 
    + \veps\sigma_{\mathrm{a}}u_{\ell}^{0}
    + \left(\frac{\sigma_{\mathrm{t}}}{\veps}-\veps\sigma_{\mathrm{a}}\right)(1-\delta_{\ell0})u_{\ell}^{0}=\veps f_{\ell}^{0}.
\end{multline}
This equation, together with the equations for \(\kappa>0\)
\begin{multline}\label{eq:rte_scale_Pn_2D_y_scalar_m>0}
    \frac{1}{2}\partial_x\left(-C_{\ell-1}^{\kappa-1}u_{\ell-1}^{\kappa-1} + D_{\ell+1}^{\kappa-1}u_{\ell+1}^{\kappa-1} + E_{\ell-1}^{\kappa+1}u_{\ell-1}^{\kappa+1} - F_{\ell+1}^{\kappa+1}u_{\ell+1}^{\kappa+1}\right) \\ 
    +\partial_z\left(A_{\ell-1}^{\kappa}u_{\ell-1}^{\kappa} + B_{\ell+1}^{\kappa}u_{\ell+1}^{\kappa}\right) 
    + \veps\sigma_{\mathrm{a}}u_{\ell}^{\kappa}
    + \left(\frac{\sigma_{\mathrm{t}}}{\veps}-\veps\sigma_{\mathrm{a}}\right)u_{\ell}^{\kappa}=\veps f_{\ell}^{\kappa},\; \kappa>0
\end{multline}
decouples the unknowns \(\bv =\begin{bmatrix}
		\left(\bu^{0}\right)^\tran & \left(\bu^{1}\right)^\tran & \cdots & \left(\bu^{N}\right)^\tran 
	\end{bmatrix}^\tran\)
from the rest and can be solved independently. \Cref{eq:rte_scale_Pn_2D_y_scalar_m=0,eq:rte_scale_Pn_2D_y_scalar_m>0} can be reformatted as
\begin{equation}\label{eq:rte_scale_Pn_2D_y_reduced}
	\bB^{(1)}\partial_x\bv + \bB^{(3)}\partial_z\bv + \veps\sigma_{\mathrm{a}}\bv + \left(\frac{\sigma_{\mathrm{t}}}{\veps}-\veps\sigma_{\mathrm{a}}\right) \bR\bv = \veps \bbf.
\end{equation}
For the $P_3$ model used in Section \ref{sect:5.3}, $\bv = \begin{bmatrix} u_0^0 & u_1^0 & u_2^0 & u_3^0 & u_1^1 & u_2^1 & u_3^1 & u_2^2 & u_3^2 & u_3^3 \end{bmatrix}^\tran$ is the unknown and the matrices  $\bB^{(1)}$ and $\bB^{(3)}$ are given in the sparse format by \Cref{tab:B_1,tab:B_3}, respectively. Note that \(\bB^{(1)}\) is no longer symmetric because of this trick to decouple the negative \(\kappa\) unknowns. 
We remark that the unknowns for the \(P_N\) equation for the 2-D plane-parallel model are nearly halved by the additional symmetry due to the reduction of one spatial variable. 

\begin{table}[h]
    \centering
    \begin{tabular}{c|c c c c c c}
       \toprule
       \multicolumn{7}{c}{\(\bB^{(1)}\)}\\
       \midrule
       \((i,j)\) & \((1,5)\) & \((2,6)\) & \((3,5)\) & \((3,7)\) & \((4,6)\) & \((5,1)\) \\
       Value & \(-F_1^1\) & \(-F_2^1\) & \(E_1^1\) & \(-F_3^1\) & \(E_2^1\) & \(-C_0^0/2\) \\
       \midrule
       \((i,j)\) & \((5,3)\) & \((5,8)\) & \((6,2)\) & \((6,4)\) & \((6,9)\) & \((7,3)\)\\
       Value & \(D_2^0/2\) & \(-F_2^2/2\) & \(-C_1^0/2\) & \(D_3^0/2\) & \(-F_3^2/2\) & \(-C_2^0/2\) \\
       \midrule
       \((i,j)\)  & \((7,8)\) & \((8,5)\) & \((8,7)\) & \((8,10)\) & \((9,6)\) & \((10,8)\) \\
       Value & \(E_2^2/2\) & \(-C_1^1/2\) & \(D_3^1/2\) & \(-F_3^3/2\) & \(-C_2^1/2\) & \(-C_2^2/2\) \\
       \bottomrule
    \end{tabular}
    \caption{Element values of \(\bB^{(1)}\) in row \(i\) and column \(j\)}\label{tab:B_1}
\end{table}
\begin{table}[h]
    \centering
    \begin{tabular}{c|c c c c c c}
       \toprule
       \multicolumn{7}{c}{\(\bB^{(3)}\)}\\
       \midrule
       \((i,j)\) & \((1,2)\) & \((2,1)\) & \((2,3)\) & \((3,2)\) & \((3,4)\) & \((4,3)\) \\
       Value & \(B_1^0\) & \(A_0^0\) & \(B_2^0\) & \(A_1^0\) & \(B_3^0\) & \(A_2^0\) \\
       \midrule
       \((i,j)\) & \((5,6)\) & \((6,5)\) & \((6,7)\) & \((7,6)\) & \((8,9)\) & \((9,8)\)\\
       Value & \(B_2^1\) & \(A_1^1\) & \(B_3^1\) & \(A_2^1\) & \(B_3^2\) & \(A_2^2\) \\
       \bottomrule
    \end{tabular}
    \caption{Element values of \(\bB^{(3)}\) in row \(i\) and column \(j\)}\label{tab:B_3}
\end{table}


\section{Hybrid discontinuous Galerkin methods} \label{sec:appx}
In this appendix, we derive hybrid spherical harmonics DG methods for the stationary \(2\)-D plane-parallel model \eqref{eq:2d} with rectangular elements.  
The hybrid DG algorithm is applied to the \(2\)-D \(P_N\) equation of the form \eqref{eq:rte_scale_Pn_2D_y_reduced} in the domain $X=[a,b]\times[c,d]$, and can be applied to other models with obvious modifications. 


Let $a=x_{1/2}<x_{3/2}<\cdots < x_{i+{1}/{2}} <\cdots< x_{M_x + {1}/{2}}=b$ and $c=z_{{1}/{2}}<z_{{3}/{2}}<\cdots < z_{j+{1}/{2}}<\cdots < z_{M_z + {1}/{2}}=d$. Denote by $\mathcal{T}^h$ the rectangular partition of $X$, that is
\begin{equation}
\mathcal{T}^h = \left\{K_{i,j}:=I_i\times J_j:=[x_{i-\frac{1}{2}},x_{i+\frac{1}{2}}]\times[z_{j-\frac{1}{2}},z_{j+\frac{1}{2}}]\colon 1\leq i\leq M_x,\,\, 1\leq j \leq M_z\right\}.
\end{equation}
For simplicity, we assume the uniform mesh for each direction $x$ and $z$ and denote by $h_x$ and $h_z$ the length of the intervals on the $x$- and $z$-axis, respectively.

Let $\bu^h_0\in \bV^{h,1}_{0}$ be a $\mathbb{Q}_1$ function in each cell $K_{i,j}$, i.e.,
\begin{equation}\label{eq:expr_sol_2d}
	\bu^h_{0}|_{K_{i,j}} 
 = \overline{\bu}_{0,i,j}
 + \widehat{\bu}_{0,i,j} \xi_i 
 + \widecheck{\bu}_{0,i,j} \eta_j 
 + \widetilde{\bu}_{0,i,j} \xi_i\eta_j,
\end{equation}
where $\xi_i=\frac{2(x-x_i)}{h_x}\in [-1,1]$, $\eta_j=\frac{2(z-z_j)}{h_z}\in [-1,1]$.   For $\ell>0$, let $\bu^h_\ell\in \bV^{h,0}_1$ be a $\mathbb{P}_0$ function in each cell $K_{i,j}$, i.e.,
\begin{equation}
\bu^{h}_\ell|_{K_{i,j}} = \overline{\bu}_{\ell,i,j}, \quad \ell=1,\cdots,N.
\end{equation}
Integration of \eqref{eq:rte_scale_Pn_2D_y_reduced}  over $K_{ij}$ gives 
	\begin{multline}
		\frac{1}{h_xh_z}\int_{J_{j}}\left(\overrightarrow{\bB^{(1)}\bu^h}_{i+\frac{1}{2},j} -\overrightarrow{\bB^{(1)}\bu^h}_{i-\frac{1}{2},j}\right) \ud z \\
  + \frac{1}{h_xh_z}\int_{I_{i}}\left(\overrightarrow{\bB^{(3)}\bu^h}_{i,j+\frac{1}{2}} -\overrightarrow{\bB^{(3)}\bu^h}_{i,j-\frac{1}{2}}\right) \ud x 
		+ \veps\sigma_{\mathrm{a}}\overline{\bu^h}_{i,j} + \left(\frac{\sigma_{\mathrm{t}}}{\veps}-\veps\sigma_{\mathrm{a}}\right)\bR\overline{\bu^h}_{i,j} = \veps \overline{\bbf}_{i,j},
	\end{multline}
Setting $\overrightarrow{\bB^{(1)}\bu^h}_{i+\frac{1}{2},j}
= \bB^{(1)} \avg{\bu^h}_{i+\frac{1}{2},j}-\frac{1}{2}\bD^{(1)}\jmp{\bu^h}_{i+\frac{1}{2},j}$ 
and 
$\overrightarrow{\bB^{(3)}\bu^h}_{i,j+\frac{1}{2}}
= \bB^{(3)} \avg{\bu^h}_{i,j+\frac{1}{2}} -\frac{1}{2}\bD^{(3)} \jmp{\bu^h}_{i,j+\frac{1}{2}}$, where \(\bD^{1} =|\bB|^{1}\) and \(\bD^{3} =|\bB|^{3}\) are positive definite matrices, gives
	\begin{multline}
 \label{eq:rte_dg_mixed_2d}
		\frac{1}{h_xh_z}\int_{J_{j}}\bigg(\Big(\bB^{(1)}\avg{\bu^h}_{i+\frac{1}{2},j} -\frac{1}{2}\bD^{(1)}\jmp{\bu^h}_{i+\frac{1}{2},j}\Big) \\
        - \Big(\bB^{(1)}\avg{\bu^h}_{i-\frac{1}{2},j} + \frac{1}{2}\bD^{(1)}\jmp{\bu^h}_{i-\frac{1}{2},j}\Big)\bigg)\ud z \\
		+ \frac{1}{h_xh_z}\int_{I_{i}}\bigg(\Big(\bB^{(3)}\avg{\bu^h}_{i,j+\frac{1}{2}} -\frac{1}{2}\bD^{(3)}\jmp{\bu^h}_{i,j+\frac{1}{2}}\Big) \\
        - \Big(\bB^{(3)}\avg{\bu^h}_{i,j-\frac{1}{2}} + \frac{1}{2}\bD^{(3)}\jmp{\bu^h}_{i,j-\frac{1}{2}}\Big)\bigg)\ud x \\
		+ \veps\sigma_{\mathrm{a}}\overline{\bu}_{i,j} + \left(\frac{\sigma_{\mathrm{t}}}{\veps}-\veps\sigma_{\mathrm{a}}\right)\bR\overline{\bu}_{i,j} = \veps \overline{\bbf}_{i,j}.
	\end{multline}

The edge values required to define the jumps and averages in \eqref{eq:rte_dg_mixed_2d} are determined using \eqref{eq:expr_sol_2d} for $\ell=0$ and local reconstructions for $\ell >0$.  In the $\ell=0$ case, equations for $\widehat{\bu}_{0,i}$, $\widecheck{\bu}_{0,i}$, $\widetilde{\bu}_{0,i}$ are required.  To derive them, we multiply the $\ell=0$ component of \eqref{added3}  by $\xi_i$, $\mu_i$, and $\mu_i \xi_i$, respectively, and integrate over $K_{i,j}$.  
With the definition of the numerical fluxes above, the results are
\begin{subequations}
	\begin{multline}
		\frac{3}{h_xh_z}\int_{J_{j}} \bigg(\Big(\bB^{(1)}\avg{\bu^h}_{i+\frac{1}{2},j} -\frac{1}{2}\bD^{(1)}\jmp{\bu^h}_{i+\frac{1}{2},j}\Big)_0 \\
		+\Big(\bB^{(1)}\avg{\bu^h}_{i-\frac{1}{2},j} + \frac{1}{2}\bD^{(1)}\jmp{\bu^h}_{i-\frac{1}{2},j}\Big)_0\bigg)\ud z
		-\frac{6}{h_x}(\bB^{(1)}\overline{\bu}_{i,j})_0  \\
		+ \frac{3}{h_xh_z}\int_{I_{i}}\xi_i \bigg(\Big(\bB^{(3)}\avg{\bu^h}_{i,j+\frac{1}{2}} -\frac{1}{2}\bD^{(3)}\jmp{\bu^h}_{i,j+\frac{1}{2}}\Big)_0 \\
		-\Big(\bB^{(2)}\avg{\bu^h}_{i,j-\frac{1}{2}} + \frac{1}{2}\bD^{(3)}\jmp{\bu^h}_{i,j-\frac{1}{2}}\Big)_0\bigg)\ud x
		+ \veps\sigma_{\mathrm{a}} (\widehat{\bu}_{i,j})_0 
		= \veps (\widehat{\bbf}_{i,j})_0,
	\end{multline}
	\begin{multline}
		\frac{3}{h_xh_z}\int_{J_{j}}\eta_j \bigg(\Big(\bB^{(1)}\avg{\bu^h}_{i+\frac{1}{2},j} -\frac{1}{2}\bD^{(1)}\jmp{\bu^h}_{i+\frac{1}{2},j}\Big)_0 \\
		-\Big(\bB^{(1)}\avg{\bu^h}_{i-\frac{1}{2},j} + \frac{1}{2}\bD^{(1)}\jmp{\bu^h}_{i-\frac{1}{2},j}\Big)_0\bigg)\ud z -\frac{6}{h_z}(\bB^{(3)}\overline{\bu}_{i,j})_0
  \\
		+ \frac{3}{h_xh_z}\int_{I_{i}} \bigg(\Big(\bB^{(3)}\avg{\bu^h}_{i,j+\frac{1}{2}} -\frac{1}{2}\bD^{(3)}\jmp{\bu^h}_{i,j+\frac{1}{2}}\Big)_0 \\
		+\Big(\bB^{(2)}\avg{\bu^h}_{i,j-\frac{1}{2}} + \frac{1}{2}\bD^{(3)}\jmp{\bu^h}_{i,j-\frac{1}{2}}\Big)_0\bigg)\ud x
		+ \veps\sigma_{\mathrm{a}} (\widecheck{\bu}_{i,j})_0
		= \veps (\widecheck{\bbf}_{i,j})_0,
	\end{multline}
	\begin{multline}
		\frac{9}{h_xh_z}\int_{J_{j}}\eta_j \bigg(\Big(\bB^{(1)}\avg{\bu^h}_{i+\frac{1}{2},j} -\frac{1}{2}\bD^{(1)}\jmp{\bu^h}_{i+\frac{1}{2},j}\Big)_0 \\
		+\Big(\bB^{(1)}\avg{\bu^h}_{i-\frac{1}{2},j} + \frac{1}{2}\bD^{(1)}\jmp{\bu^h}_{i-\frac{1}{2},j}\Big)_0\bigg) \ud z
		- \frac{18}{(h_x)^2h_z}\int_{K_{i,j}}\eta_j(\bB^{(1)} \bu)_0\ud\bx\\
		+\frac{9}{h_xh_z}\int_{I_{i}}\xi_i \bigg(\Big(\bB^{(3)}\avg{\bu^h}_{i,j+\frac{1}{2}} -\frac{1}{2}\bD^{(1)}\jmp{\bu^h}_{i,j+\frac{1}{2}}\Big)_0 \\
		+\Big(\bB^{(3)}\avg{\bu^h}_{i,j-\frac{1}{2}} + \frac{1}{2}\bD^{(3)}\jmp{\bu^h}_{i,j-\frac{1}{2}}\Big)_0\bigg) \ud x \\
		- \frac{18}{h_x(h_z)^2}\int_{K_{i,j}}\xi_i(\bB^{(3)} \bu)_0\ud\bx
		+ \veps\sigma_{\mathrm{a}} (\widetilde{\bu}_{i,j})_0 
		= \veps (\widetilde{\bbf}_{i,j})_0.
	\end{multline}
\end{subequations}
where, as in the slab geometry case, we often use an outer subscript to extract the $\ell=0$ component of a vector, e.g., $(\bB^{(3)}\overline{\bu}_{i,j})_0 = \bB^{(3)}\overline{\bu}_{0,i,j}$.
Using \eqref{eq:expr_sol_2d} leads to jumps and averages
\begin{subequations}\label{eq:trace_dg_2d}
	\begin{align}
		\jmp{\bu^h_0}_{i+\frac{1}{2},j} &= \overline{\bu}_{0,i+1,j} - \widehat{\bu}_{0,i+1,j} + \widecheck{\bu}_{0,i+1,j} \eta_j - \widetilde{\bu}_{0,i+1,j} \eta_j
        \\
		&{}\hspace{0.5in}
        - \left(\overline{\bu}_{0,i,j} + \widehat{\bu}_{0,i,j} + \widecheck{\bu}_{0,i,j} \eta_j + \widetilde{\bu}_{0,i,j} \eta_j \right),
        \nonumber\\
		\jmp{\bu^h_0}_{i-\frac{1}{2},j} 
        &= -\overline{\bu}_{0,i,j} + \widehat{\bu}_{0,i,j} - \widecheck{\bu}_{0,i,j} \eta_j + \widetilde{\bu}_{0,i,j} \eta_j\\
		&{}\hspace{0.5in} + \overline{\bu}_{0,i-1,j} + \widehat{\bu}_{0,i-1,j} + \widecheck{\bu}_{0,i-1,j} \eta_j + \widetilde{\bu}_{0,i-1,j} \eta_j,
        \nonumber\\
		\avg{\bu^h_0}_{i+\frac{1}{2},j} &= \frac{1}{2}\left(\overline{\bu}_{0,i+1,j} - \widehat{\bu}_{0,i+1,j} + \widecheck{\bu}_{0,i+1,j} \eta_j - \widetilde{\bu}_{0,i+1,j} \eta_j \right)
        \\&{}\hspace{0.5in} 
        + \frac{1}{2}\left(\overline{\bu}_{0,i,j} + \widehat{\bu}_{0,i,j} + \widecheck{\bu}_{0,i,j} \eta_j + \widetilde{\bu}_{0,i,j} \eta_j \right),
        \nonumber\\
		\jmp{\bu^h_0}_{i,j+\frac{1}{2}} &= \overline{\bu}_{0,i,j+1} + \widehat{\bu}_{0,i,j+1}\xi_i - \widecheck{\bu}_{0,i,j+1}  - \widetilde{\bu}_{0,i,j+1} \xi_i\\
		&{}\hspace{0.5in}- \left(\overline{\bu}_{0,i,j} + \widehat{\bu}_{0,i,j}\xi_i + \widecheck{\bu}_{0,i,j} + \widetilde{\bu}_{0,i,j} \xi_i \right),
        \nonumber\\
		\jmp{\bu^h}_{0,i,j-\frac{1}{2}}
        &= -\overline{\bu}_{0,i,j} - \widehat{\bu}_{0,i,j}\xi_i + \widecheck{\bu}_{0,i,j}  + \widetilde{\bu}_{0,i,j} \xi_i \\
		&{}\hspace{0.5in} + \left(\overline{\bu}_{0,i,j-1} + \widehat{\bu}_{0,i,j-1} \xi_i  + \widecheck{\bu}_{0,i,j-1} + \widetilde{\bu}_{0,i,j-1} \xi_i \right),
        \nonumber\\
		\avg{\bu^h_0}_{i,j+\frac{1}{2}} 
        &= \frac{1}{2}(\overline{\bu}_{0,i,j+1} + \widehat{\bu}_{0,i,j+1}\xi_i - \widecheck{\bu}_{0,i,j+1}  - \widetilde{\bu}_{0,i,j+1} \xi_i\big)\\
		&{}\hspace{0.5in} + \frac{1}{2}\left(\overline{\bu}_{0,i,j} + \widehat{\bu}_{0,i,j}\xi_i + \widecheck{\bu}_{0,i,j} + \widetilde{\bu}_{0,i,j} \xi_i \right).
  \nonumber
	\end{align}
\end{subequations}

For $\ell>0$, we only know the cell average value of $\bu_\ell^h$. As in the slab geometry case, we employ Fromm's method to approximate slopes with center differences to find edge values and then apply upwind fluxes at cell interfaces. Let us assume a uniform mesh size $h$. Applying Fromm's method to $\bu_\ell$ on each cell gives  
\begin{equation}
\bu^h_\ell= \overline{\bu}_{\ell,i,j} + \frac{\overline{\bu}_{\ell,i+1,j} - \overline{\bu}_{\ell,i-1,j}}{4}\xi_i + \frac{\overline{\bu}_{\ell,i,j+1} - \overline{\bu}_{\ell,i,j-1}}{4}\eta_j.
\end{equation}
Therefore the edge values for $\ell>0$ are
 \begin{subequations}\label{eq:fromm_2d_1}
\begin{align}
	\bu^-_{\ell,i+\frac{1}{2},j} 
    &= \overline{\bu}_{\ell,i,j} + \frac{\overline{\bu}_{\ell,i+1,j} - \overline{\bu}_{\ell,i-1,j}}{4} + \frac{\overline{\bu}_{i,j+1} - \overline{\bu}_{\ell,i,j-1}}{4}\eta_j,
    \\
	\bu^+_{\ell,i+\frac{1}{2},j} 
     &= \overline{\bu}_{\ell,i+1,j} - \frac{\overline{\bu}_{\ell,i+2,j} - \overline{\bu}_{\ell,i,j}}{4} 
	+ \frac{\overline{\bu}_{\ell,i+1,j+1} - \overline{\bu}_{\ell,i+1,j-1}}{4}\eta_j,\\
	\bu^-_{\ell,i,j+\frac{1}{2}} 
    &= \overline{\bu}_{\ell,i,j} + \frac{\overline{\bu}_{\ell,i+1,j} - \overline{\bu}_{\ell,i-1,j}}{4}\xi_i + \frac{\overline{\bu}_{\ell,i,j+1} - \overline{\bu}_{\ell,i,j-1}}{4},\\
	\bu^+_{\ell,i,j+\frac{1}{2}} &= \overline{\bu}_{\ell,i,j+1} + \frac{\overline{\bu}_{\ell,i+1,j+1} - \overline{\bu}_{\ell,i-1,j+1}}{4}\xi_i 
    - \frac{\overline{\bu}_{\ell,i,j+2} - \overline{\bu}_{\ell,i,j}}{4},
\end{align}
\end{subequations}
and the average and jump values for $\ell>0$ are
 \begin{subequations}\label{eq:fromm_2d}
 	\begin{align}
		\jmp{\bu^h_\ell}_{i+\frac{1}{2},j} &= \frac{-\overline{\bu}_{\ell,i+2,j} + 3\overline{\bu}_{i+1,j} - 3\overline{\bu}_{i,j} + \overline{\bu}_{\ell,i-1,j}}{4} \nonumber\\
		&{}\qquad +\frac{\overline{\bu}_{\ell,i+1,j+1} - \overline{\bu}_{\ell,i,j+1} - \overline{\bu}_{\ell,i+1,j-1} +\overline{\bu}_{\ell,i,j-1}}{4}\eta_j,\\
		\jmp{\bu^h_\ell}_{i-\frac{1}{2},j} &= \frac{ \overline{\bu}_{\ell,i+1,j} - 3\overline{\bu}_{\ell,i,j} + 3\overline{\bu}_{\ell,i-1,j} - \overline{\bu}_{\ell,i-2,j}}{4} \nonumber\\
		&{}\qquad -\frac{\overline{\bu}_{\ell,i,j+1} - \overline{\bu}_{\ell,i-1,j+1} - \overline{\bu}_{\ell,i,j-1} +\overline{\bu}_{\ell,i-1,j-1}}{4}\eta_j,\\
		\avg{\bu^h_\ell}_{i+\frac{1}{2},j} &= \frac{-\overline{\bu}_{\ell,i+2,j} + 5\overline{\bu}_{\ell,i+1,j} + 5\overline{\bu}_{\ell,i,j} - \overline{\bu}_{\ell,i-1,j}}{8} \nonumber\\
		&{}\qquad +\frac{\overline{\bu}_{\ell,i+1,j+1} + \overline{\bu}_{\ell,i,j+1} - \overline{\bu}_{\ell,i+1,j-1} - \overline{\bu}_{\ell,i,j-1}}{8}\eta_j,\\
		\jmp{\bu^h_\ell}_{i,j+\frac{1}{2}} &= \frac{-\overline{\bu}_{\ell,i,j+2} + 3\overline{\bu}_{\ell,i,j+1} - 3\overline{\bu}_{\ell,i,j} + \overline{\bu}_{\ell,i,j-1}}{4} \nonumber\\
		&{}\qquad +\frac{\overline{\bu}_{\ell,i+1,j+1} - \overline{\bu}_{\ell,i+1,j} - \overline{\bu}_{i-1,j+1} +\overline{\bu}_{\ell,i-1,j}}{4}\xi_i,\\
		\jmp{\bu^h_\ell}_{i,j-\frac{1}{2}} &= \frac{ \overline{\bu}_{\ell,i,j+1} - 3\overline{\bu}_{\ell,i,j} + 3\overline{\bu}_{\ell,i,j-1} - \overline{\bu}_{\ell,i,j-2}}{4} \nonumber\\
		&{}\qquad -\frac{\overline{\bu}_{i+1,j} - \overline{\bu}_{\ell,i+1,j-1} - \overline{\bu}_{\ell,i-1,j} +\overline{\bu}_{\ell,i-1,j-1}}{4}\xi_i,\\
		\avg{\bu^h_\ell}_{i,j+\frac{1}{2}} &= \frac{-\overline{\bu}_{\ell,i,j+2} + 5\overline{\bu}_{\ell,i,j+1} + 5\overline{\bu}_{\ell,i,j} - \overline{\bu}_{\ell,i,j-1}}{8} \nonumber\\
		&{}\qquad +\frac{\overline{\bu}_{\ell,i+1,j+1} + \overline{\bu}_{\ell,i+1,j} - \overline{\bu}_{\ell,i-1,j+1} - \overline{\bu}_{\ell,i-1,j}}{8}\xi_i.
 	\end{align}
 \end{subequations}
Together \cref{eq:rte_dg_mixed_2d,eq:trace_dg_2d,eq:fromm_2d}, comprise the mixed spherical harmonics DG method in the 2-D geometry setting.

\ifx\SIAM\TRUE
\bibliographystyle{siamplain}
\else
\bibliographystyle{amsplain}
\fi
\bibliography{refs}

\end{document}